%% file: main.tex
\documentclass{article}
\usepackage[letterpaper,margin=1.0in]{geometry}

\usepackage{hyperref}
\usepackage{enumitem}
\usepackage[utf8]{inputenc} 
\usepackage[T1]{fontenc}    
\usepackage{hyperref}       
\usepackage{url}            
\usepackage{booktabs}       
\usepackage{amsfonts}       
\usepackage{nicefrac}       
\usepackage{microtype}      
\usepackage{graphicx}       
\usepackage[numbers,sort]{natbib}
\usepackage{multirow}
\usepackage{bbding}
\graphicspath{{media/}}     
\usepackage{color}
\usepackage{algorithmic,algorithm}
\usepackage{nicefrac}
\usepackage{cancel}
\usepackage{graphicx} 
\usepackage{booktabs}
\usepackage{makecell}
\usepackage{multirow}
\usepackage{amssymb,amsthm}
\input{math}

\usepackage{thmtools}
\usepackage{thm-restate}

\title{Solving Convex-Concave Problems with $\tilde{\mathcal{O}}(\epsilon^{-4/(3p+1)})$ $p$th-Order Oracle Complexity \textsuperscript{*}}

\author{Lesi Chen \textsuperscript{* 1} \qquad Xinliang Zhang \textsuperscript{* 1} \qquad Chengchang Liu \textsuperscript{2}\\
Junru Li \textsuperscript{1} \qquad Luo Luo \textsuperscript{3}
\qquad Jingzhao Zhang \textsuperscript{1} \\
\vspace{-2mm} \\
\normalsize{\textsuperscript{1}IIIS, Tsinghua University\quad 
\textsuperscript{2} The Chinese University of Hong Kong
}
\\ \normalsize{\textsuperscript{3} Fudan University}\\
   \vspace{-2mm} \\
\normalsize{ \texttt{ \{chenlc23,xinliang23\}@mails.tsinghua.edu.cn, 7liuchengchang@gmail.com,}} \\
\normalsize{\texttt{
jr-li24@mails.tsinghua.edu.cn,
luoluo@fudan.edu.cn,
jingzhaoz@mail.tsinghua.edu.cn
}}}

\begin{document}
\maketitle
\begingroup
\begin{NoHyper}
\renewcommand\thefootnote{*}
\footnotetext{The first two authors contribute equally. The conference version of this manuscript is published in COLT 2025 \citep{chen2025solving}, which contains the upper bound when $p=2$. This work shows our upper bound can be readily extended to $p$th-order methods, and provides a new lower bound for the problem.}
\end{NoHyper}
\endgroup
\begin{abstract}
When the objective has Lipschitz continuous $p$th-order derivatives, it is known that
convex-concave minimax problems can be solved with $\mathcal{O}(\epsilon^{-2/(p+1)})$ $p$th-order oracle calls. This complexity upper bound was speculated to be optimal as it is achieved by a natural generalization of the optimal first-order method.
In this work, we show an improved upper bound of
$\tilde{\mathcal{O}}(\epsilon^{-4/(3p+1)})$ 
by applying the Monteiro-Svaiter acceleration.
We also establish a lower complexity bound of $\Omega(\epsilon^{-2/(3p-1)})$, suggesting a gap still exists for $p \ge 2$.
\end{abstract}

\section{Introduction}
We study smooth convex-concave minimax optimization problems over convex compact sets $\gX \subset \sR^{d_x}$ and $\gY \subset \sR^{d_y}$:
\begin{equation} \label{prob:main}
\min_{\vx \in \gX} \max_{\vy \in \gY} f(\vx, \vy),
\end{equation}
where the objective $f$ has $L_p$-Lipschitz continuous $p$th-order derivatives.

This problem naturally arises from many applications, including finding a Nash equilibrium of a two-player zero-sum game~\citep{v1927hilbertschen,nash1950equilibrium,jordan2025adaptive}, solving the Lagrangian function in constrained optimization \citep{ouyang2021lower,scaman2017optimal}, and many machine learning problems such as adversarial training \citep{zhang2018mitigating}, AUC maximization~\citep{ying2016stochastic}, and distributionally robust optimization~\citep{carmon2022distributionally}. 

A fundamental result for Problem (\ref{prob:main}) is the minimax theorem, which was first proved for bilinear problems by \citet{v1927hilbertschen}, and extended to general convex-concave functions by \citet{sion1958general}. The minimax theorem guarantees the existence of a saddle point $(\vx^*,\vy^*)$ such that 
\begin{equation*}
    f(\vx^*,\vy^*) = \min_{\vx \in \gX} \max_{\vy \in \gY} f(\vx,\vy) = \max_{\vy \in \gY}  \min_{\vx \in \gX} f(\vx,\vy).
\end{equation*}

Given the existence of the saddle point, a natural question is: 
``What is the complexity of finding such a point?'' We study this question under the oracle model \citep{nemirovskij1983problem}. This model restricts algorithms to access oracles that contain local information about the objectives, where the $p$th-order oracle returns the tuple $(f(\vx,\vy), \nabla f(\vx,\vy), \cdots, \nabla^p f(\vx,\vy))$ that gives the function value and derivatives up to order $p$ at the query point $(\vx,\vy)$. 
The order $p$ naturally gives a division of different algorithm classes. 
For example, first-order methods purely use gradient information; second-order methods (also known as Newton-type methods) jointly use gradient and Hessian information; and tensor methods $(p \ge 2)$ leverage higher-order derivatives information.
A central problem in optimization theory research is to accurately characterize the gains provided by higher-order oracles~\citep{gasnikov2019near,nesterov2021inexact,nesterov2021superfast}.

The complexity of first-order methods for solving minimax optimization is well understood. 
In his seminal work,
\citet{korpelevich1976extragradient} proposed the celebrated extragradient (EG) method and proved it can find an $\epsilon$-solution to Problem (\ref{prob:main}) in $\gO\left( \nicefrac{L_1}{\epsilon} \right)$ iterations, where each iteration calls two first-order oracles.
\citet{popov1980modification} proposed the optimistic gradient (OG) method which has the same convergence guarantee but only requires a single first-order call every iteration. \citet{nemirovski2004prox, nesterov2007dual} further proposed methods for more general spaces known as mirrox-prox (MP) and dual extrapolation (DE), respectively. \citet{nemirovski2004prox} also proved a lower bound to show that the 
$\Omega( \nicefrac{L_1}{\epsilon} )$ complexity is the limit of all first-order algorithms, and recent works~\citep{zhang2022lower,ouyang2021lower} further refined the diameter dependency in the lower bound.



Compared with the theory of first-order methods, we have less knowledge of higher-order minimax optimization.
In their seminal work, \citet{monteiro2012iteration} proposed the
second-order generalization of EG, which can provably find an $\epsilon$-solution with  $\gO( (\nicefrac{L_2}{\epsilon})^{2/3})$ second-order calls.
Recently, 
\citet{bullins2022higher} generalized this method to non-Euclidean setup and to $p$th order, achieving the $\gO( (\nicefrac{L_p}{\epsilon})^{2/(p+1)})$ upper complexity bound of $p$th-order oracle calls.
\citet{huang2022approximation,adil2022optimal} simplified the proof of the $p$th-order EG by viewing the $p$th-order tensor step as an oracle.
\citet{lin2023monotone} provided alternative analyses from the
control-theoretic perspective.
Besides the high-order EG algorithm, numerous other algorithms have also been proposed, including high-order OG \citep{jiang2022generalized,jiang2024adaptive}, high-order DE \citep{lin2022perseus}, and high-order reduced-gradient methods \citep{nesterov2023high}.
All of these methods have the same $\gO((\nicefrac{L_p}{\epsilon}) ^{2/(p+1)})$ upper complexity bound of $p$th-order oracle calls, which is speculated to be optimal in \citep{adil2022optimal,lin2022perseus}.

The main contribution of this paper is to break the commonly conjectured 
$\Omega((\nicefrac{L_p}{\epsilon})^{2/(p+1)})$ barrier by proving a better upper complexity bound of $\tilde \gO( (\nicefrac{L_p}{\epsilon})^{4/(3p+1)})$ for objectives with $L_p$-Lipschitz continuous $p$th-order derivatives. Our new result is achieved by applying the high-order momentum acceleration technique \citep{monteiro2013accelerated} to both variables $\vx$ and $\vy$, and reducing solving the original Problem~(\ref{prob:main}) to solving $\tilde \gO\left((\nicefrac{\gamma}{\epsilon})^{4/(3p+1)} \right)$ proximal minimax subproblems 
\begin{equation*}
    \min_{\vx \in \gX} \max_{\vy \in \gY} f(\vx,\vy) + \frac{\gamma}{p+1} \Vert \vx - \bar \vx \Vert^{p+1} - \frac{\gamma}{p+1} \Vert \vy - \bar \vy \Vert^{p+1}, 
\end{equation*}
where $(\bar \vx, \bar \vy)$ is the proximal center. If we take $\gamma = L_p$, the above subproblem has a constant condition number for $p$th-order methods. Then an  $\tilde \gO\left( (\nicefrac{L_p}{\epsilon})^{4/(3p+1)} \right)$ upper complexity bound can easily be achieved by applying any existing convergent $p$th-order method to solve the above subproblems.

Finally, we complement our upper bound by proving a lower complexity bound showing that any $p$th-order algorithm can not find an $\epsilon$-solution in less than $\Omega( (\nicefrac{L_p}{\epsilon})^{2/(3p-1)} )$ oracle calls. Our lower bound is the first non-trivial lower bound for high-order minimax optimization under general conditions:
Existing lower bounds either require the diameter to scale with $\nicefrac{1}{\epsilon}$ \citep{adil2022optimal,lin2022perseus}, or is subsumed by the  $\Omega( (\nicefrac{1}{\epsilon})^{2/(3p+1)})$ bound for standard minimization problems~\citep{arjevani2019oracle} (see the discussion in Section~\ref{sec:diss-lower}).
Our lower bound matches the upper bound when $p=1$, but a gap still exists when $p \ge 2$. 
We hope our results can be an initial step for future works to further fill the gap.

\paragraph{Additional Related Works} Apart from the aforementioned related works in minimax optimization, we also review closely relevant studies on high-order convex optimization. In a seminal work, \citet{nesterov2006cubic} proposed the Cubic Regularized Newton (CRN) method with the $\gO( (\nicefrac{L_2}{\epsilon})^{1/2})$ second-order oracle complexity, which initiated a series of subsequent works analyzing the complexity of second- and higher-order algorithms.
\citet{nesterov2008accelerating} proposed the accelerated CRN (ACRN) method that achieves the improved 
$\gO( (\nicefrac{L_2}{\epsilon})^{1/3})$ complexity. 
\citet{monteiro2013accelerated} introduced a novel and intricate acceleration framework, under which they proposed a method with $\gO( (\nicefrac{L_2}{\epsilon})^{2/7} \log (\nicefrac{1}{\epsilon}))$ second-order oracle complexity. 
\citet{nesterov2021implementable}
gave a second-order implementation of third-order tensor step and 
generalized the ACRN method to $p$th order, achieving the $\gO( (\nicefrac{L_p}{\epsilon})^{1/(p+1)} )$ upper bound.
Later, three independent groups \citep{jiang2019optimal,bubeck2019near,gasnikov2019optimal} generalized the Monteiro-Svaiter accelerated method to $p$th order and achieved the $\gO( (\nicefrac{L_p}{\epsilon})^{2/(3p+1)} \log (\nicefrac{1}{\epsilon}))$ upper bound, which matches the $\Omega((\nicefrac{L_p}{\epsilon})^{2/(3p+1)})$ lower bound established by \citet{arjevani2019oracle} up to only logarithmic factors. Recently, \citet{carmon2022optimal,kovalev2022first} independently proposed different novel enhancements to remove the $\log (\nicefrac{1}{\epsilon})$ factor caused by the line search subroutine in the Monteiro-Svaiter acceleration, achieving the optimal $\gO((\nicefrac{L_p}{\epsilon})^{2/(3p+1)})$ oracle complexity.
\citet{carmon2020acceleration} analyzed both the upper and lower bounds for the limiting behavior of the $p$th-order Monteiro-Svaiter acceleration when $p \rightarrow \infty$, which can be used to improve the complexity of many related optimization problems \citep{carmon2021thinking,carmon2024whole,carmon2022distributionally}.

\paragraph{Notations} We use $\Vert \, \cdot \, \Vert$ to denote the Euclidean norm for vectors and the spectral norm for matrices. We hide logarithmic factors in the notation $\tilde \gO(\,\cdot\,)$.
For a set $\gZ \subseteq \sR^d$, we let $\gN_\gZ(\vz) = \left\{\vv \in \sR^d: \langle \vv, \vz' - \vz \rangle \le 0, \forall \vz' \in \gZ \right\}$ be the normal cone of $\gZ$ at point $\vz$. We also let ${\rm Proj}_{\gZ}(\bar \vz ) = \arg \min_{\vz \in \gZ} \Vert \vz - \bar \vz \Vert$ be the projection operator of $\vz$ onto the set $\gZ$. 
To simplify notations, we let $\vz = (\vx,\vy)$ and define the operator on $\vz$ as 
$\mF = (\nabla_x f, - \nabla_y f)^\top$. Let  $\mT$ and $\mM$ be two order $k$ tensors, their inner product is defined as $\langle \mT, \mM \rangle = \sum_{i_1,\cdots,i_k} \mT_{i_1,\cdots,i_k} \mM_{i_1,\cdots,i_k} $. 
Notation $\otimes$ denotes the Kronecker product. For a $k$ order tensor $\mT$, we define its spectral norm as 
\begin{equation*}
    \Vert \mT \Vert = \sup_{\vv_1,\cdots,\vv_k} \left\{ \langle \mT, \vv_1 \otimes \cdots \otimes \vv_k \rangle   : \Vert \vv_i \Vert \le 1, 1 \le i \le k \right\}.
\end{equation*}

\section{Preliminaries} \label{sec:pre}
In this section, we describe the formal problem setup. We begin with assumptions made in our analyses.
First, we assume that the domain is bounded and that the function is convex-concave.
\begin{asm}[Bounded domain] \label{asm:D}
    We assume that both the sets $\gX$ and $\gY$ are convex and compact with diameters $D_\gX:= \sup_{\vx, \vx' \in \gX} \Vert \vx - \vx' \Vert < + \infty$ and  $D_\gY:= \sup_{\vy, \vy' \in \gY} \Vert \vy - \vy' \Vert < + \infty$.
    We also define $\gZ = \gX \times \gY$ and $D_\gZ := \sup_{\vz, \vz' \in \gZ} \Vert \vz - \vz' \Vert$.
\end{asm}
\begin{asm}[Convex-concave functions] \label{asm:CC}
    We assume that $f(\,\cdot\, \vy)$ is convex for any fixed $\vy$ and $f(\vx,\,\cdot\,)$ is concave for any fixed $\vx$.
\end{asm}
Recall that $\vz= (\vx,\vy)$ and that $\mF = (\nabla_x f, -\nabla_y f)^\top$. The convex-concavity assumption is also equivalent to assuming the monotonicity of the operator $\mF$:
\begin{equation*}
    \langle \mF(\vz) - \mF(\vz'), \vz - \vz' \rangle \ge 0, \quad \forall \vz, \vz' \in \gZ.
\end{equation*}
Third, we assume that the function has the following smoothness properties. The first one is standard for analyses of first-order algorithms. The last one provides a guarantee for higher-order methods.
\begin{asm}[Gradient Lipschitzness] \label{asm:grad-lip}
    We suppose the gradient of $f(\vx,\vy)$ is $L_1$-Lipschitz continuous for some $L_1>0$:
    \begin{equation*}
        \Vert \nabla f(\vx,\vy) - \nabla f(\vx',\vy') \Vert \le L_1 \Vert \vx - \vy \Vert ,\quad \forall \vx,\vx' \in \gX, ~\vy,\vy' \in \gY.
    \end{equation*}
\end{asm}
\begin{asm}[$p$th-order Derivative Lipschitzness] \label{asm:Hess-lip}
    We suppose $\nabla^p f(\vx,\vy)$ is $L_p$-Lipschitz continuous for some $L_p >0$:
    \begin{equation*}
        \Vert \nabla^p f(\vx,\vy) - \nabla^p f(\vx',\vy') \Vert \le L_p \Vert \vx - \vy \Vert ,\quad \forall \vx,\vx' \in \gX, ~\vy,\vy' \in \gY.
    \end{equation*}
\end{asm}

We notice that previous $p$th-order EG method \citep{monteiro2012iteration,bullins2022higher} only require Assumption \ref{asm:CC} and \ref{asm:Hess-lip}, but unnecessarily require Assumption~\ref{asm:D} and \ref{asm:grad-lip}. In the following, we discuss that our two additional assumptions only incur mild complexity cost.

\begin{remark}[Discussion on Assumption \ref{asm:D}]
Although the $p$th-order EG also works under the unconstrained setting, the convergence rates will depend on $\Vert \vz_0 - \vz^* \Vert$. When we know $\bar D_\gZ$ is an upper bound for $\Vert \vz_0 - \vz^* \Vert$, we can simply run our algorithm on the constraint set $ \{ \vz \in \sR^{d} : \Vert \vz - \vz_0 \Vert \le \bar D_\gZ \}$, where $d = d_x+d_y$. When we underestimate $\bar D_\gZ$, we can run the algorithm with $\bar D_\gZ$ doubled each time until we find an $\epsilon$-solution, which only causes an additional $\log( \Vert \vz_0 - \vz^* \Vert / \bar D_\gZ )$ factor in the total complexity. 
\end{remark}

\begin{remark}[Discussion on Assumption \ref{asm:grad-lip}]
Although the additional Assumption \ref{asm:grad-lip} makes our algorithm slightly more restricted than the accelerated algorithms such as $p$th-order EG, the upper complexity bound of our algorithm will only have a mild logarithmic dependency on~$L_1$.  Moreover, Assumption \ref{asm:grad-lip} can be derived from Assumption \ref{asm:Hess-lip} when the domains $\gX,\gY$ are compact (Assumption~\ref{asm:D}). 
\end{remark}


\subsection{High-Order Oracle} 
We then introduce the oracle for measuring complexity. Following \citet{nemirovskij1983problem}, we restrict the scope of this paper to algorithms that use the $p$th-order oracles defined below. 
\begin{dfn}[$p$th-order oracle]
Given a query point $(\vx,\vy) \in \gX \times \gY$, the $p$th-order oracle returns the following derivative information
\begin{equation*}
    \nabla^{(0,\cdots,p)} f(\vx,\vy):= (f(\vx,\vy), \nabla f(\vx,\vy), \cdots, \nabla^p f(\vx,\vy)).
\end{equation*}
\end{dfn}
To leverage this $p$th-order oracle, a common way is to take a tensor step
\citep{gasnikov2019near,bullins2022higher} that is formally defined as follows. 
In this work, we only consider the case that $p$ is a constant and leave the study of limiting behavior when $p \rightarrow \infty$ \citep{carmon2021thinking} as future work.
\begin{dfn}[$p$th-order tensor step] \label{dfn:tensor-step}
For an input point $\bar \vz \in \gZ$ and a monotone operator $\mF : \sR^d \rightarrow \sR^d$ the $p$th-order tensor step outputs $\vz = \gT_p(\bar \vz; \mF, M)$ such that
\begin{equation} \label{eq:Tensor-step}
    \left \langle \bar \mF_{p-1}(\vz; \bar \vz) + \frac{M}{p!} \Vert \vz - \bar \vz \Vert^{p-1} (\vz - \bar \vz)  , \vz' - \vz \right \rangle \ge 0, \quad \forall \vz' \in \gZ,
\end{equation}
where $\bar \mF_{p-1}(\vz; \bar \vz): = \sum_{k=0}^{p-1} \frac{1}{i!} \nabla^k \mF(\bar \vz) [\vz- \bar \vz]^k$ is the $(p-1)$th-order Taylor expansion of $\mF$.
\end{dfn}

The algorithm proposed in this work jointly uses the above $p$th-order tensor step with $\mF = (\nabla_x f, - \nabla_y f)^\top$, $\mF = \nabla_x f$, and $\mF = -\nabla_y f$ in different iterates. 
When $p=1$, the above tensor step is exactly the (projected) gradient step; When $p=2$, it corresponds to the cubic regularized Newton step \citep{nesterov2006cubic} and
can be efficiently solved via binary search \citep{monteiro2012iteration}; In general ($p \ge 2$), it corresponds to solving a strongly monotone variational inequality when setting $M >L_p$, which can be solved by the interior point method \citep{ralph2000superlinear,qi2002smoothing} or the cutting plane method \citep{jiang2020improved}. 

\subsection{Optimality Criterion}
A natural goal for the optimality of a solution $z$ in Problem (\ref{prob:main}) is the duality gap \citep{nesterov2007dual}:
\begin{equation} \label{dfn:duality-gap}
    {\rm Gap}_f(\vx, \vy): = \max_{\vy' \in \gY} f(\vx, \vy') - \min_{\vx' \in \gX} f(\vx', \vy).
\end{equation}
In this work, we instead adopt a criterion that is stronger than the 
widely used duality gap, defined as follows. 
\begin{dfn}[{\citet{cai2022finite}}] \label{dfn:eps-sol} 
We say $\vz \in \gZ $ is an $\epsilon$-solution to Problem (\ref{prob:main}) if 
\begin{equation*}
   r_\mF^{\rm tan}(\vz) := \min_{\vc \in \gN_\gZ(\vz)} \Vert \mF(\vz) + \vc \Vert \le \epsilon, \quad {\rm where} ~~ \mF = (\nabla_x f, - \nabla_y f)^\top.
\end{equation*}
\end{dfn}
The above criterion is known as the tangent residual, which is introduced by \citet{cai2022finite} to generalize the gradient norm from unconstrained to constrained settings.
It is lower-bounded by the duality gap, as formally stated in the following lemma.
\begin{lemma} \label{lem:solu-concept}
Under Assumption \ref{asm:CC}, we have that
\begin{equation*}
    {\rm Gap}_f (\vx,\vy) 
    \le D_\gZ \min_{\vc \in \gN_{\gZ}} \Vert \mF(\vz) + \vc \Vert, \quad \forall \vz = (\vx,\vy) \in \gX \times \gY.
\end{equation*}
\end{lemma}
\begin{proof}
It follows from Cauchy-Schwarz inequality and the convex-concavity of function $f(\vx,\vy)$: For $\vc_m = \arg \min_{\vc  \in \gN_{\gZ}} \Vert \mF(\vz) + \vc \Vert$, we have that
\begin{align*}
    & {\rm Gap}_f (\vx,\vy)  = \max_{\vy' \in \gY} f(\vx,\vy') - \min_{\vx' \in \gX} f(\vx', \vy) \\
\le& \max_{\vy' \in \gY} \langle \nabla_y f(\vx, \vy), \vy' - \vy \rangle + \min_{\vx' \in \gX} \langle \nabla_x f(\vx,\vy), \vx- \vx' \rangle \\
=& \max_{\vz' \in \gZ} \langle \mF(\vz'), \vz - \vz' \rangle \le \langle \mF(\vz) + \vc_m, \vz - \vz' \rangle  \\
\le& D_\gZ \Vert \mF(\vz) + \vc \Vert.
\end{align*}    
\end{proof}
To find a point with a small tangent residual, we exploit smoothness and bound the tangent residual by the distance to the optimal solution in the following.
\begin{lemma} \label{lem:gd-make-g-small}
Let the operator $\mF: \sR^d \rightarrow \sR^d$ be monotone and $L$-Lipschitz continuous. For any $\vz^* \in \sR^d$ such that $r_\mF^{\rm tan}(\vz^*) = 0$, one gradient step $\vz = {\rm Proj}_\gZ (\bar \vz - \eta \mF(\bar \vz))$ ensures
\begin{equation*}
    r_{\mF}^{\rm tan}(\vz) \le \left \Vert \frac{\bar \vz - \vz}{\eta} + \mF(\vz) - \mF(\bar z)   \right \Vert \le 
    \left( \frac{1}{\eta} +L  \right) ( 2 + \eta L )
    \Vert \bar \vz - \vz^* \Vert.
\end{equation*}
Consequently, setting $\eta = 1/L$ yields $r_{\mF}^{\rm tan}(\vz)  \le 6 L \Vert \bar \vz - \vz^* \Vert$. 
\end{lemma}
\begin{proof}
From the update $\vz = {\rm Proj}_\gZ (\bar \vz - \eta \mF(\bar \vz))$ we know that
\begin{equation*}
    \frac{\bar \vz - \eta \mF(\bar \vz) - \vz
    }{\eta} = \frac{\bar \vz - \vz}{\eta} - \mF(\bar \vz) \in \gN_\gZ(\vz).
\end{equation*}
Therefore, we can upper-bound the tangent residual as follows.
\begin{align*}
    r_{\mF}^{\rm tan}(\vz) &\le  \left \Vert  (\bar \vz - \vz)/\eta + \mF(\vz) - \mF(\bar z)   \right \Vert \le \left( 1/\eta +L  \right) \Vert \vz - \bar \vz \Vert \\
    &=  \left( 1/\eta +L  \right) \left \Vert {\rm Proj}_\gZ ( \bar \vz - \eta \mF(\bar \vz)) - \bar \vz \right \Vert \\
    &= \left( 1/\eta +L  \right) \left \Vert {\rm Proj}_\gZ ( \bar \vz - \eta \mF(\bar \vz)) - {\rm Proj}_\gZ ( \vz^* - \eta \mF(\vz^*)) + \vz^* - \bar \vz \right \Vert \\
    &\le  \left( 1/\eta +L  \right) ( 2 + \eta L ) \Vert \bar \vz -\vz^* \Vert,
\end{align*}
where the last line uses the non-expansiveness of the projection operator, the $L$-Lipschitz continuity of operator $\mF$, and the triangle inequality.  
\end{proof}

\subsection{Uniform Convexity} Finally, we introduce uniform convexity, which defines a natural function class that $(p-1)$th-order tensor methods enjoy a linear convergence rate \citep{gasnikov2019optimal}. This function class provides an important tool in our analyses.



\begin{dfn} \label{dfn:UC}
    A function $h(\vz): \gZ \rightarrow \sR$ is 
    $p$th-order 
    $\mu$-uniformly convex for some $\mu>0$ if 
    \begin{equation*}
        h(\vz) \ge h(\vz') + \langle \nabla h(\vz'), \vz - \vz' \rangle + \frac{\mu}{p} \Vert \vz - \vz' \Vert^p, \quad \forall \vz, \vz' \in \gZ.
    \end{equation*}
We say $h(\vz)$ is $p$th-order  $\mu$-uniformly concave if $-h(\vz)$ is $p$th-order  $\mu$-uniformly convex.
\end{dfn}
In the following, we recall some properties of uniformly convex functions, which are also useful to derive the main results of this paper.

\begin{lemma}[Section 4.2.2, \citet{nesterov2018lectures}] \label{lem:UC-grad-dominant}
Let $h(\vz): \gZ \rightarrow \sR$ be $p$th-order $\mu$-uniformly convex. Then for any $\vz \in \gZ$ and $\vz^* = \arg \min_{\vz \in \gZ} h(\vz)$, we have that 
\begin{equation*}
        \frac{p-1}{p} \left( \frac{1}{\mu} \right)^{\frac{1}{p-1}} \Vert \nabla h(\vz) \Vert^{\frac{p}{p-1}} \ge h(\vz) - h(\vz^*) \ge \frac{\mu}{p} \Vert \vz - \vz^* \Vert^p, \quad \forall \vz \in \gZ.
\end{equation*}
\end{lemma}
A typical example of uniformly convex functions is the power function
$d_p(\vz) = \Vert \vz \Vert^p / p$, whose properties are listed in the following lemma.
\begin{lemma}\label{lem:cubic-func}
The function $d_p(\vz) =  \Vert \vz \Vert^p / p$ with $p \in \sN_+$ satisfies the following:
\begin{enumerate}
    \item It is $p$th-order $\left(1/2 \right)^{p-2}$-uniformly convex;
    \item It has $(p-1)!$-Lipschitz continuous $p$th-order derivatives.
\end{enumerate}
\end{lemma}

\begin{proof} See  \citep[Lemma 4.2.3]{nesterov2018lectures} for item 1, and \citep[Theorem 7.1]{rodomanov2020smoothness} for item 2, respectively.
  
\end{proof}


Next, let us define uniformly-convex-uniformly-concave functions.

\begin{dfn}
We say $f(\vx,\vy)$ is $p$th-order $\mu_x$-uniformly-convex-$\mu_y$-uniformly-concave if $f(\vx,\,\cdot\,)$ is $p$th-order $\mu_x$-uniformly convex for any fixed $\vx \in \sR^{d_x}$ and $f(\,\cdot\,\vy)$ is $p$th-order $\mu_y$-uniformly concave for any fixed $\vy \in \sR^{d_y}$, where $\mu_x, \mu_y >0$. 
\end{dfn}

The following lemma shows that the gradient operator of a $\mu$-uniformly-convex-$\mu$-uniformly-concave function is $({2\mu}/{p})$-uniformly monotone.

\begin{lemma} \label{lem:U-monotone}
If $f: \gX \times \gY \rightarrow \sR$ is a $p$th-order $\mu$-uniformly-convex-$\mu$-uniformly concave function, then its gradient operator is $({2\mu}/{p})$-uniformly monotone, \textit{i.e.}, 
\begin{equation} \label{eq:U-monotone}
    \langle \vg_1 - \vg_2, \vz_1- \vz_2 \rangle \ge \frac{2\mu}{p} \Vert \vz_1 - \vz_2 \Vert^{p}, \quad \forall \vg_1 \in \mA(\vz_1), \vg_2 \in \mA(\vz_2),~~ \forall \vz_1,\vz_2 \in \gZ,
\end{equation}
where $\mA(\vz) = \mF(\vz) + \gN_\gZ(\vz)$ and $\mF = (\nabla_x f, - \nabla_y f)^\top$.
\end{lemma}
\begin{proof} For any $ (\vx_1,\vy_2),(\vx_2,\vy_2)\in \gX \times \gY$, we have that
\begin{align*}
    f(\vx_2,\vy_1) - f(\vx_1,\vy_1) &\ge \langle \nabla_x f(\vx_1,\vy_1) + \vu_1, \vx_2 - \vx_1 \rangle + \frac{\mu}{p} \Vert \vx_2 - \vx_1 \Vert^{p}, ~~ \vu_1 \in \gN_{\gX} (\vx_1); \\
    f(\vx_1,\vy_2) - f(\vx_2,\vy_2) &\ge \langle \nabla_x f(\vx_2, \vy_2) + \vu_2, \vx_1 - \vx_2 \rangle + \frac{\mu}{p} \Vert \vx_2 - \vx_1 \Vert^{p}, ~~ \vu_2 \in \gN_{\gX} (\vx_2); \\
    -f(\vx_1,\vy_2) + f(\vx_1,\vy_1) &\ge \langle -\nabla_y f(\vx_1,\vy_1) +\vv_1, \vy_2 - \vy_1 \rangle + \frac{\mu}{p} \Vert \vy_2 - \vy_1 \Vert^{p}, ~~\vv_1\in \gN_{\gY}(\vy_1);  \\
    -f(\vx_2,\vy_1) + f(\vx_2,\vy_2) &\ge  \langle -\nabla_y f(\vx_2,\vy_2) + \vv_2, \vy_1 - \vy_2 \rangle + \frac{\mu}{p} \Vert \vy_2 - \vy_1 \Vert^{p}, ~~ \vv_2 \in \gN_{\gY} (\vy_2).
\end{align*}
Summing up the above four inequalities yields Eq. (\ref{eq:U-monotone}).  
\end{proof}

\begin{cor} \label{cor:U-monotone}
If $f: \gX \times \gY \rightarrow \sR$ is a $p$th-order $\mu$-uniformly-convex-$\mu$-uniformly concave function, then for any $\vg \in \mA(\vz)$ and $\vz^* \in \{\vz \in \sR^d: 0 \in \mA(\vz)  \}$ we have that 
    \begin{equation*}
        \Vert \vg \Vert \ge \frac{2 \mu}{p} \Vert \vz - \vz^* \Vert^{p-1}.
    \end{equation*}
\end{cor}

\section{Accelerated Inexact Proximal Extragradient} \label{sec:AIPE}

Before presenting our main algorithm, we first introduce the Accelerated Inexact Proximal Extragradient (AIPE) method for minimizing a convex function $h(\vz): \gZ \rightarrow \sR$, which
will be an important component in our main algorithm. We present its procedure in Algorithm \ref{alg:ANPE-restart}, which uses the inexact oracles to
generalize the recently proposed optimal Monteiro-Svaiter acceleration (OptMS) \citep{carmon2022optimal}, and applies the restart scheme \citep{gasnikov2019optimal} to exploit uniform convexity. 
The lines using inexact oracles are marked in blue in Algorithm \ref{alg:ANPE-restart}.

\begin{dfn}[Inexact Zeroth-order Oracle]
    We call ${\rm iFunc}_h$ a $\delta$-zeroth-order oracle of function $h:\gZ \rightarrow \sR$ if for every $\vz \in \gZ$ it returns $\hat h = {\rm iFunc}_h(\vz,\delta)$ satisfying~$\vert \hat h- h(\vz)  \vert \le \delta$.
\end{dfn}

\begin{algorithm*}[t]  
\caption{AIPE-restart$(\vz^{(0)}, \gamma, \delta, T, S) $
{\color{gray} $\triangleright$ A meta algorithm}
}\label{alg:ANPE-restart}
\begin{algorithmic}[1] 
\renewcommand{\algorithmicrequire}{ \textbf{Input:}}
\FOR{$s=0,\cdots,S-1$}
\STATE $\vv_0^{(s)} = \bar \vz_0^{(s)} = \vz^{(s)}_0 = \vz^{(s)}$, $A_0 = 0$, $\lambda_1'^{(s)} = 1$ \\
\STATE {\color{gray}$\triangleright$ Run MS acceleration with inexact proximal and gradient oracle in each epoch.} \\
\FOR{$t = 0,\cdots, T-1$}  
\STATE \quad {\color{blue}$ \hat h_t^{(s)} = {\rm iFunc}_h(\vz_t^{(s)},\delta)$, $\tilde h_t^{(s)} = {\rm iFunc}_h(\tilde \vz_t^{(s)},\delta)$ } {\color{gray} \hfill $\triangleright$ Record the function values.} \label{Line:ifunc} \\ 
\STATE \quad 
Solve ${a_{t+1}'^{(s)}} > 0 $ from $ A_t^{(s)} +  a_{t+1}'^{(s)} = {\color{blue}2} \lambda_{t+1}'^{(s)} \left(a_{t+1}'^{(s)}\right)^2$ and let $A_{t+1}'^{(s)} = A_t^{(s)} + a_{t+1}'^{(s)}$ \label{Line:eq-a}\\
\STATE \quad $\bar \vz_t^{(s)} = \frac{A_t^{(s)}}{A_{t+1}'^{(s)}} \vz_t^{(s)} + \frac{a_{t+1}'^{(s)}}{A_{t+1}'^{(s)}} \vv_t^{(s)} $  \label{Line:eq-a-end} \\
\STATE \quad $\tilde \vz_{t+1}^{(s)}, \vu_{t+1}^{(s)}= {\color{blue}{\rm iProx}_h (\bar \vz_t^{(s)}, \gamma, \delta)} $ \label{Line:iprox} {\color{gray} $\triangleright$ Inexact $p$th-order proximal update.} %
\STATE \quad $\lambda_{t+1}^{(s)} = \gamma \Vert \tilde \vz_{t+1}^{(s)} - \bar \vz_t^{(s)} \Vert^{p-1}$
\STATE \quad {\color{gray} $\triangleright$ Adaptively search $\lambda_{t+1}'^{(s)} \approx \lambda_{t+1}^{(s)}$.} 
\STATE \quad \textbf{if} $t = 0 $ \textbf{then}  \label{line:update-lmbd}
\STATE \quad \quad  $\lambda_{t+1}'^{(s)} = \lambda_{t+1}^{(s)}$\\
\STATE \quad \textbf{end if} \\
\STATE \quad  \textbf{if} $\lambda_{t+1}^{(s)} \le \lambda_{t+1}'^{(s)}$ \textbf{then}  \\
\STATE \quad \quad $\gamma_{t+1}^{(s)}=1$, $a_{t+1}^{(s)} = a_{t+1}'^{(s)} $, $A_{t+1}^{(s)} = A_{t+1}'^{(s)}$, $\lambda_{t+2}'^{(s)} = \frac{1}{2} \lambda_{t+1}'^{(s)}$\\
\STATE \quad \textbf{else} \\
\STATE \quad \quad $\gamma_{t+1}^{(s)} = \frac{\lambda_{t+1}'^{(s)}}{\lambda_{t+1}^{(s)}}$, $a_{t+1}^{(s)} = \gamma_{t+1}^{(s)} a_{t+1}'^{(s)}$, $A_{t+1}^{(s)} = A_t^{(s)} + a_{t+1}^{(s)}$, $\lambda_{t+2}'^{(s)} = 2 \lambda_{t+1}'^{(s)}$ \\
\STATE \quad \textbf{end if} \label{line:update-lmbd-end} \\
\STATE \quad $ \vz_{t+1}^{(s)} = \frac{(1- \gamma_{t+1}^{(s)}) A_t^{(s)}}{A_{t+1}^{(s)}} \vz_t^{(s)} + \frac{\gamma_{t+1}^{(s)} A_{t+1}'^{(s)}}{A_{t+1}^{(s)}} \tilde \vz_{t+1}^{(s)}$ \\
\STATE \quad {\color{blue} $\vg_{t+1}^{(s)} = {\rm iGrad}_h(\tilde \vz_{t+1}^{(s)},\delta)$} \label{Line:igrad} \\
\STATE \quad $\vv_{t+1}^{(s)} = {\rm Proj}_{\gZ} (\vv_t^{(s)} - a_{t+1}^{(s)} (\vg_{t+1}^{(s)} + \vu_{t+1}^{(s)}))$ \hfill {\color{gray}$\triangleright$ Extragradient update with inexact gradients.} \\ 
\ENDFOR \\
\STATE 
{\color{gray} $\triangleright$
Select $\vz^{s+1}$ as the point that achieves the lowest value of $\{ \hat h_t^{(s)},\tilde h_t^{(s)} \}_{t=0}^{T-1}$.}
\STATE  
$\hat t = \arg \min_{0 \le t \le T} \hat h_t^{(s)}$, $\tilde t = \arg \min_{0 \le t \le T} \tilde h_t^{(s)}$ \\ \label{line:select-best-beg}
\STATE \textbf{if} $\hat h_{\hat t}^{(s)} < \tilde h_{\tilde t}^{(s)}$  \textbf{then}\\
\STATE \quad $\hat \vz^{(s+1)} = \hat \vz_t^{(s)}$ \\
\STATE \textbf{else} \\
\STATE \quad $\tilde \vz^{(s+1)} = \tilde \vz_t^{(s)}$ \\
\STATE \textbf{end if} \\ \label{line:select-best-end}
\ENDFOR \\
\RETURN $\vz^{S}$
\end{algorithmic}
\end{algorithm*}

\begin{dfn}[Inexact First-order Oracle]
 We call ${\rm iGrad}_h$ a $\delta$-first-order oracle of function $h:\gZ \rightarrow \sR$ if for every $\vz \in \gZ$ the oracle returns 
 $\vg = {\rm iGrad}_h(\vz,\delta)$ satisfying $\Vert \vg 
 - \nabla h(\vz) \Vert \le \delta$.
\end{dfn}

\begin{dfn}[Inexact $p$th-order Proximal Oracle] \label{dfn:inexact-MS-oracle}
We call ${\rm iProx}_h$ 
a $(\delta,\gamma)$-$p$th-order proximal oracle
for function $h: \gZ \rightarrow \sR$ if for every $ \bar \vz \in \gZ$ and $\gamma>0$ the oracle returns $(\vz, \vu) = {\rm iProx}_h(\bar \vz, \gamma, \delta)$ with $\vz \in \gZ$ and $\vu \in \gN_{\gZ}(\vz)$ satisfying
\begin{equation} \label{eq:cond-appr-prox}
    \Vert  \nabla h(\vz) + \vu + \lambda(\vz - \bar \vz) \Vert \le \frac{\lambda}{2} \Vert \vz-  \bar \vz \Vert+ \delta, \quad {\rm where} \quad \lambda = \gamma \Vert \vz - \bar \vz \Vert^{p-1}. 
\end{equation}
\end{dfn}

When $\delta=0$, Definition \ref{dfn:inexact-MS-oracle} reduces to the oracle used in previous works \citep{monteiro2013accelerated,carmon2022optimal,kovalev2022first,nesterov2021inexact,nesterov2023inexact}, which can be implemented by a tensor step, as shown in the following lemma.

\begin{lemma}[{\citet[Section 3.1]{carmon2022optimal}}] \label{lem:CRN-is-MS}
If $h(\vz): \gZ \rightarrow \sR$ satisfies Assumption~\ref{asm:Hess-lip}, then taking a $p$th-order tensor step  according to Definition \ref{dfn:tensor-step} with regularization $M= 2L_p$
implements an $( 0, L_p)$-proximal oracle.
\end{lemma}

As a warm-up, we first give the theoretical guarantee of the exact algorithm for uniformly convex functions, which can be simply obtained by combining the guarantee of OptMS algorithm \citep{carmon2022optimal} with the standard analysis for restart schemes \citep{arjevani2019oracle,gasnikov2019optimal}.

\begin{theorem}[OptMS-restart] \label{thm:ANPE}
Assume $h(\vz): \gZ \rightarrow \sR$ is $(p+1)$th-order $\mu$-uniformly convex and has $L_p$-Lipschitz continuous $p$th-order derivatives.
Then running Algorithm \ref{alg:ANPE-restart} with ${\rm iProx}_h$ implemented with $p$th-order tensor step according to Lemma \ref{lem:CRN-is-MS} and $\gamma = L_p$
returns a point~$\vz \in \gZ$ such that $h(\vz^S) - h(\vz^*) \le \epsilon$ in $\gO\left( (\nicefrac{L_p}{\mu})^{2/(3p+1)} \log (\nicefrac{\Delta}{\epsilon})  \right)$ $p$th-order oracle calls.
\end{theorem}

\begin{proof}
It is implied by Theorem \ref{thm:AIPE} with $\delta=0$.  
\end{proof}


Next, we state the following general result that allows for inexact oracle errors.

\begin{theorem}[AIPE-restart] \label{thm:AIPE}
Assume that $h(\vz): \gZ \rightarrow \sR$ is $(p+1)$th-order $\mu$-uniformly convex. Let $\vz^* = \arg \min_{\vz \in \gZ} h(\vz)$, $\Delta = h(\vz^0) - h(\vz^*)$, and $D_\gZ = \sup_{\vz, \vz' \in \gZ}\Vert \vz- \vz' \Vert$. If
\begin{equation} \label{eq:cond-delta}
    \delta \le \min \left\{
    \frac{ (1 +\sqrt{17}) \epsilon^2 \min \left\{D_\gZ, \left(\frac{(p+1)\epsilon^2}{2 \mu \Delta}\right)^{\frac{1}{p+1}} \right\}}{2^7 (7+ \sqrt{6})D_\gZ^2  \Delta  }, \frac{\epsilon^2}{2^4 \max\{ D_\gZ, 1 \}  \Delta }  \right\}
\end{equation}
then running Algorithm~\ref{alg:ANPE-restart} with
$S = \lceil \log_2 (\nicefrac{\Delta}{\epsilon}) \rceil$ and $T = \gO\left( (\nicefrac{\gamma}{\mu})^{2/(3p+1)} \right)$ returns $\vz^{S}$ such that $h(\vz^S) - h(\vz^*) \le \epsilon$ in total $TS = \gO\left( \left( (\nicefrac{\gamma}{\mu})^{2/(3p+1)} \right) \log (\nicefrac{\Delta}{\epsilon}) \right) $ iterations.
\end{theorem}

\begin{proof}
See Appendix \ref{apx:proof-AIPE} for the proof.  
\end{proof}

\begin{remark}
Some of our proof for handling inexactness is motivated by the proof of \citep[Theorem 7]{bubeck2019complexity}. But our analysis is simpler as we use the search-free  Monteiro-Svaiter acceleration~\citep{carmon2022optimal}, so we do not need to analyze the additional line search with inexact proximal oracles \citep[Section E]{bubeck2019complexity}. Besides the conciseness, another advantage of our analysis is that it can also tackle the constrained case. In contrast, 
the line search procedure \citep[Lemma 31]{bubeck2019complexity} requires the function to be twice differentiable, which is violated by the indicator function $\gI_{\gZ}$ under the constrained case.
\end{remark}


For convex optimization, both AIPE and OptMS match the lower bound for their own algorithm class. Specifically, \citet[Theorem 3]{arjevani2019oracle} proved that OptMS is an optimal $p$th-order method, while recently \citet[Theorem 1.4]{adil2024convex} proved the AIPE with $\delta=0$ is an optimal $p$th-order proximal method. Both of them will be important 
subroutines in our accelerated algorithm.

\section{The $\tilde \gO(\epsilon^{-4/(3p+1)})$ Upper Bound} \label{sec:main}


In this section, we introduce our accelerated method Minimax-AIPE for Problem~(\ref{prob:main}). First, Section \ref{subsec:reduction} shows that solving the original problem can be reduced to solving a $\mu_x$-uniformly-convex-$\mu_y$-uniformly-concave minimax problem by adding small regularization. Then, Section \ref{subsec:algo} presents the formal algorithm and proves that it can find an $\epsilon$-solution to a convex-concave minimax problem in $\tilde \gO(\epsilon^{-4/(3p+1)})$ $p$th-order oracle calls.




\subsection{A Uniformly-Convex-Uniformly-Concave Surrogate} \label{subsec:reduction}

Instead of directly optimizing the original objective,  our accelerated algorithm optimized the following uniformly-convex-uniformly-concave surrogate function 
\begin{equation} \label{eq:dfn-f-eps}
    f_{\epsilon}(\vx,\vy) := f(\vx,\vy) + \frac{\mu_x}{p+1} \Vert \vx - \vx_0 \Vert^{p+1} - \frac{\mu_y}{p+1} \Vert \vy - \vy_0 \Vert^{p+1},
\end{equation}
where we set $\mu_x = \epsilon / (4 D_\gX^p)$ and $\mu_y = \epsilon / (4D_\gY^p)$ to ensure that the gradients of the objective before and after regularization are  $(\epsilon/2)$-close to each other, as shown in the following.
\begin{lemma} \label{lem:reg-f-eps-close}
Under Assumption  \ref{asm:D} and \ref{asm:CC}, the regularized function in Eq. (\ref{eq:dfn-f-eps}) satisfies:
\begin{enumerate}
\item $f_\epsilon(\vx,\vy)$ is $(p+1)$th-order $({\mu_x}/{2^{p-1}})$-uniformly-convex-$({\mu_y}/{2^{p-1}})$-uniformly-concave.
\item Let $\mu_x = \epsilon / (4 D_\gX^p)$ and $\mu_y = \epsilon / (4D_\gY^p)$. If $(\hat \vx, \hat \vy)$ is an $(\nicefrac{\epsilon}{2})$-solution to $f_\epsilon(\vx,\vy)$ as per Definition \ref{dfn:eps-sol}, then it is an $\epsilon$-solution to $f(\vx,\vy)$.
\end{enumerate}
\end{lemma}
Therefore, after adding the above regularization, we reduce solving the original problem to solving an $\gO(\epsilon)$-uniformly-convex-$\gO(\epsilon)$-uniformly-concave function. 
The same regularization technique is also used in the development of the accelerated algorithm for $p=1$ \citep{lin2020near}. 
In the following lemma, we
generalize
\citep[Lemma B.2]{lin2020near} from $p=1$ to general $p \in \sN_+$ and show that the regularized objective 
enjoys benign properties that facilitate our analyses.
\begin{lemma} \label{lem:benigh-reg-f}
Let $\gX$ and $\gY$ be convex sets.
If the function $f(\vx,\vy)$ is $p$th-order $\mu_x$-uniformly-convex-$\mu_y$-uniformly-concave on $\gX \times \gY$,
then we have the following.
\begin{enumerate}
    \item The primal function $\Phi(\vx) = \max_{\vy \in \gY} f(\vx,\vy)$ and dual function $\Psi(\vy) = \min_{\vx \in \gX} f(\vx,\vy)$ are $p$th-order $\mu_x$-uniformly convex and $\mu_y$-uniformly concave, respectively.
    \item If there exists $L^{xy}>0$ such that $\nabla_y f(\vx,\vy)$ is $L^{xy}$-Lipschitz continuous in $\vx$ and $\nabla_x f(\vx,\vy)$ is $L^{xy}$-Lipschitz continuous in $\vy$, then the solution mappings $\vy^*(\vx) = \arg \max_{\vy \in \gY} f(\vx,\vy)$ 
    and $\vx^*(\vy) = \arg \min_{\vx \in \gX} f(\vx,\vy)$ satisfy
    \begin{align*}
        \frac{2\mu_y}{p} \Vert \vy^*(\vx_1) - \vy^*(\vx_2) \Vert^{p-1} \le& L^{xy} \Vert \vx_1  - \vx_2 \Vert, \quad \forall \vx_1,\vx_2 \in \gX;\\
        \frac{2 \mu_x}{p} \Vert \vx^*(\vy_1) - \vx^*(\vy_2) \Vert^{p-1} \le& 
        L^{xy} \Vert \vy_1  - \vy_2 \Vert, \quad \forall \vy_1,\vy_2 \in \gY;
    \end{align*}
\end{enumerate}

\end{lemma}



\begin{proof}
Below, we provide the proof for $\Phi(\vx)$ and $\vy^*(\vx)$, and the result for $\Psi(\vy)$ and $\vx^*(\vy)$ can be obtained by exchanging the variables $\vx$ and $\vy$.
\begin{enumerate}
    \item For any $\vx_1,\vx_2 \in \gX$, we have that
\begin{align*}
 &\quad   \Phi(\vx_1) - \Phi(\vx_2) - \langle \Phi(\vx_2), \vx_1 - \vx_2 \rangle \\
 &= f(\vx_1, \vy^*(\vx_1)) - f(\vx_2, \vy^*(\vx_2)) - \langle \nabla_x f(\vx_2, \vy^*(\vx_2)), \vx_1 - \vx_2 \rangle \\
 &\ge f(\vx_1, \vy^*(\vx_2)) - f(\vx_2, \vy^*(\vx_2)) - \langle \nabla_x f(\vx_2, \vy^*(\vx_2)), \vx_1 - \vx_2 \rangle \\
 &\ge \frac{\mu_x}{p} \Vert \vx_1 - \vx_2 \Vert^{p},
\end{align*}
where the last step uses the assumption that $f(\,\cdot\,,\vy)$ is $p$th-order $\mu_x$-uniformly convex. This proves the $p$th-order $\mu_x$-uniform convexity by definition.
    \item 
    For any $\vx_1,\vx_2 \in \gX$, the optimality condition of $\vy^*(\vx)$ indicates that
\begin{align*}
   \langle \nabla_y f(\vx_1, \vy^*(\vx_1)),  \vy - \vy^*(\vx_1) \rangle &\le 0, \quad \forall \vy \in \gY; \\
   \langle \nabla_y f(\vx_2, \vy^*(\vx_2)), \vy' - \vy^*(\vx_2) \rangle &\le 0,\quad \forall \vy' \in \gY.
\end{align*}
Plugging $\vy = \vy^*(\vx_2) $ and $\vy' = \vy^*(\vx_1)$ and summing up the above inequalities gives
\begin{equation*}
    \langle \nabla_y f(\vx_1, \vy^*(\vx_1)) - \nabla_y f(\vx_2, \vy^*(\vx_2)),  \vy^*(\vx_2) - \vy^*(\vx_1) \rangle  \le 0.
\end{equation*}
By the $\mu_y$-uniform concavity in $\vy$, we know from Lemma \ref{lem:U-monotone} that 
{\begin{equation*}
    \langle \nabla_y f(\vx_1, \vy^*(\vx_2)) - \nabla_y f(\vx_1, \vy^*(\vx_1)),  \vy^*(\vx_2) - \vy^*(\vx_1) \rangle \le - \frac{2\mu_y}{p} \Vert \vy^*(\vx_2) - \vy^*(\vx_1) \Vert^{p}.  
\end{equation*}}
Summing up the above two inequalities yields that
\begin{equation*}
   \langle \nabla_y f(\vx_1, \vy^*(\vx_2)) - \nabla_y f(\vx_2, \vy^*(\vx_2)),  \vy^*(\vx_2) - \vy^*(\vx_1) \rangle \le - \frac{2\mu_y}{p} \Vert \vy^*(\vx_2) - \vy^*(\vx_1) \Vert^{p}.
\end{equation*}
Then, we apply the $L^{xy}$-Lipschitz continuity of $\nabla_y f(\vx,\vy)$ to obtain that
\begin{equation*}
    \frac{2\mu_y}{p} \Vert \vy^*(\vx_2) - \vy^*(\vx_1) \Vert^{p} \le L^{xy} \Vert \vx_1 - \vx_2 \Vert  \Vert \vy^*(\vx_2) - \vy^*(\vx_1) \Vert,
\end{equation*}
which proves the continuity of $\vy^*(\vx)$. The result for $\vx^*(\vy)$ can be proved similarly. 
\end{enumerate}
 
\end{proof}

The continuity of solution mappings  $\vy^*(\vx)$ and $\vx^*(\vy)$ (item 2 in Lemma \ref{lem:benigh-reg-f}) is 
crucial for accelerated algorithms, but not necessary for standard (high-order) EG analysis \citep{monteiro2012iteration,bullins2022higher}. 
For this reason, our algorithm additionally requires the $L_1$-smoothness assumption (Assumption \ref{asm:grad-lip}), while the high-order EG does not need. But as we have discussed, our additional assumption is mild since the $L_1$-smoothness always holds in a compact set if the function has $L_p$-Lipschitz continuous $p$th-order derivatives. 


\subsection{The Minimax-AIPE Algorithm} \label{subsec:algo}

\begin{algorithm*}[htbp]  
\caption{Minimax-AIPE \\
{\color{gray} $\triangleright$ Outer loop: Solve $\min_{\vx \in \gX} \left\{ \Phi(\vx) := \max_{\vy \in \gY }f_\epsilon(\vx,\vy)\right\}$ with $\tilde \gO((\nicefrac{\gamma}{\mu_x})^{2/(3p+1)})$ steps.} }\label{alg:Minimax-AIPE}
\begin{algorithmic}[1] 
\renewcommand{\algorithmicrequire}{ \textbf{Parameters:}} 
\REQUIRE $\gamma, S_{1},T_{1}, \delta_1$
\STATE Initialize at any $\vx^{(0)} \in \gX$.
\FOR{$s=0,\cdots,S_{1}-1$} 
\STATE $\vv_0^{(s)} = \bar \vx_0^{(s)} = \vx^{(s)}_0 = \vx^{(s)}$, $A_0 = 0$, $\lambda_1'^{(s)} = 1$
\FOR{$t = 0,\cdots, T_{1}-1$ }  
\STATE \quad {Run OptMS-restart on $f_\epsilon(\vx_t^{(s)},\,\cdot\,)$ to obtain
$ \hat \Phi_t^{(s)} = {\rm iFunc}_\Phi(\vx_t^{(s)},\delta_1)$.}  
\STATE \quad Run OptMS-restart on $f_\epsilon(\tilde \vx_t^{(s)},\,\cdot\,)$ to obtain $\tilde \Phi_t^{(s)} = {\rm iFunc}_\Phi(\tilde \vx_t^{(s)},\delta_1)$.\\
\quad {\color{gray} $\triangleright$ Both steps require $\tilde \gO((\nicefrac{L_p}{\mu_y})^{2/(3p+1)})$ $p$th-order oracles of $f$.}
\STATE \quad Update $a_{t+1}'^{(s)}$, $A_{t+1}'^{(s)}$ as Line \ref{Line:eq-a} in Algorithm \ref{alg:ANPE-restart}.
\STATE \quad $\bar \vx_t^{(s)} = \frac{A_t^{(s)}}{A_{t+1}'^{(s)}} \vx_t^{(s)} + \frac{a_{t+1}'^{(s)}}{A_{t+1}'^{(s)}} \vv_t^{(s)} $  \\
\STATE \quad Invoke
Algorithm \ref{alg:Minimax-AIPE-mid} to obtain
$\tilde \vx_{t+1}^{(s)}, \vu_{t+1}^{(s)}= {{\rm iProx}_\Phi (\bar \vx_t^{(s)}, \gamma, \delta_1)} $. \\
\quad {\color{gray} $\triangleright$ Requires $\tilde \gO((\nicefrac{\gamma}{\mu_y})^{2/(3p+1)})$ middle-loop steps.}
\STATE \quad $\lambda_{t+1}^{(s)} = \gamma \Vert \tilde \vx_{t+1}^{(s)} - \bar \vx_t^{(s)} \Vert^{p-1}$ \\
\STATE \quad Update $\gamma_{t+1}^{(s)}$, $a_{t+1}^{(s)}$, $A_{t+1}^{(s)}$, $\lambda_{t+2}'^{(s)}$ as Line \ref{line:update-lmbd} to \ref{line:update-lmbd-end} in Algorithm \ref{alg:ANPE-restart}.
\STATE \quad $ \vx_{t+1}^{(s)} = \frac{(1- \gamma_{t+1}^{(s)}) A_t^{(s)}}{A_{t+1}^{(s)}} \vx_t^{(s)} + \frac{\gamma_{t+1}^{(s)} A_{t+1}'^{(s)}}{A_{t+1}^{(s)}} \tilde \vx_{t+1}^{(s)}$ \\
\STATE \quad {Run 
OptMS-restart on $f_\epsilon(\tilde \vx_{t+1}^{(s)}, \,\cdot\,)$ to obtain
$\vg_{t+1}^{(s)} = {\rm iGrad}_\Phi(\tilde \vx_{t+1}^{(s)},\delta_1)$.} 
\\
\quad {\color{gray} $\triangleright$ Requires $\tilde \gO((\nicefrac{L_p}{\mu_y})^{2/(3p+1)})$ $p$th-order oracles of $f$.}
\STATE \quad $\vv_{t+1}^{(s)} = {\rm Proj}_\gX (\vv_t^{(s)} - a_{t+1}^{(s)} (\vg_{t+1}^{(s)} + \vu_{t+1}^{(s)}) )$ \\
\ENDFOR \\
\STATE Select $\vx^{(s+1)}$ as the point achieving the lowest value of $\{ \hat \Phi_t^{(s)},\tilde \Phi_t^{(s)} \}_{t=0}^{T-1}$. \\
\ENDFOR \\
\STATE Run OptMS-restart to find $\hat \vy$ that approximately solves $\max_{\vy \in \gY }f_\epsilon(\vx^{(S_1)},\vy)$.
\STATE {\color{gray} $\triangleright$ Take an additional gradient step to make gradient small.}
\STATE $\tilde L_1 = L_1 + p \max\{ \mu_x D_\gX^{p-1}, \mu_y D_\gY^{p-1} \}$
\STATE Let $\hat \vz = (\vx^{(S_1)},\hat \vy)$ and update $\tilde \vz = {\rm Proj}_\gZ(\hat \vz - (1/\tilde L_1) \mF(\hat \vz))$ 
\RETURN $\tilde \vz$
\end{algorithmic}
\end{algorithm*}

We introduce a triple-looped algorithm Minimax-AIPE to find an approximate solution to the uniformly-convex-uniformly-concave function $f_\epsilon(\vx,\vy)$.
Below, we introduce the procedures of each loop one by one.

The \textbf{outer} loop (Algorithm \ref{alg:Minimax-AIPE}) applies the AIPE-restart (Algorithm \ref{alg:ANPE-restart}) introduced in Section~\ref{sec:AIPE} to minimize the primal objective $\Phi(\vx):= \max_{\vy \in \gY } f_\epsilon(\vx,\vy)$, which requires the inexact zeroth-order oracles, first-order oracles, and high-order proximal oracles of $\Phi(\vx)$. As both the inexact zeroth- and first-order oracles of $\Phi(\vx)$ are easily obtainable (see Lemma \ref{lem:get-zo-fo}), the non-trivial one is the proximal oracle, which we solve in the middle loop (Algorithm~\ref{alg:Minimax-AIPE-mid}).
According to Theorem~\ref{thm:AIPE}, if the middle loop can successfully return a sufficiently accurate proximal oracle, then the outer loop (Algorithm~\ref{alg:Minimax-AIPE}) requires $\tilde \gO\left((\nicefrac{\gamma}{\mu_x})^{2/(3p-1)} \right)  $ iterations.
In this way,
Algorithm \ref{alg:Minimax-AIPE} reduces solving the original problem $\min_{\vx \in \gX} \max_{\vy \in \gY} f_\epsilon(\vx,\vy)$ to solving $\tilde \gO\left((\nicefrac{\gamma}{\mu_x})^{2/(3p+1)} \right)$ proximal oracles $ {\rm Prox}_{\Phi}(\bar \vx, \gamma) =
    \arg \min_{\vx \in \gX} \left\{ \Phi(\vx) + \gamma \Vert \vx - \bar \vx \Vert^{p+1} / (p+1) \right\}$, 
which can be obtained by solving the following one-sided regularized minimax problem
\begin{equation} \label{prob:minimax-g}
    \min_{\vx \in \gX} \max_{\vy \in \gY} \left\{ g_\epsilon(\vx,\vy; \bar \vx) := f_\epsilon(\vx,\vy) + \frac{\gamma}{p+1} \Vert \vx - \bar \vx \Vert^{p+1} \right\}.
\end{equation}
However, solving the minimax problem (\ref{prob:minimax-g}) is still challenging.
Here, our main idea is to use the same procedure again in variable $\vy$. This leads to the \textbf{middle} loop (Algorithm \ref{alg:Minimax-AIPE-mid}), which applies AIPE-restart (Algorithm~\ref{alg:ANPE-restart}) to maximize the dual objective $\Psi(\vy;\bar \vx) := \min_{\vx \in \gX} g_\epsilon(\vx,\vy; \bar \vx)$. Consequently, the middle loop reduces solving the minimax problem (\ref{prob:minimax-g}) to solving $\tilde \gO\left( (\nicefrac{\gamma}{\mu_y})^{2/(3p+1)} \right)$  proximal  oracles ${\rm Prox}_{\Psi(\,\cdot\, , \bar \vx)} = \arg \max_{\vy \in \gY} \left\{ \Psi(\vy;\bar \vx) - \gamma \Vert \vy - \bar \vy \Vert^{p+1}/ (p+1) \right\}$, which can be obtained by
solving the following two-sided regularized minimax problem
\begin{equation} \label{prob:minimax-h}
    \min_{\vx \in \gX} \max_{\vy \in \gY} \left\{ h_\epsilon(\vx,\vy; \bar \vx,\bar \vy)  := f_\epsilon(\vx,\vy) + \frac{\gamma}{p+1} \Vert \vx - \bar \vx \Vert^{p+1} - \frac{\gamma}{p+1} \Vert \vy - \bar \vy \Vert^{p+1} \right\}. 
\end{equation}
The above subproblem is well-conditioned, which can be efficiently solved 
by the \textbf{inner} loop (Algorithm \ref{alg:Minimax-AIPE-inner}).
By applying the $p$th-order EG \citep{bullins2022higher} with a restart scheme on the subproblem (\ref{prob:minimax-h}), the inner loop only requires 
$\tilde\gO\left( (\nicefrac{L_p}{\gamma})^{2/(p+1)} \right)$ $p$th-order oracle calls under Assumption~\ref{asm:Hess-lip}. Finally, setting $\gamma = L_p$ and recalling our setting in Eq. (\ref{eq:dfn-f-eps}) that $\mu_x = \epsilon / (4D_{\gX}^p)$ and $\mu_y = \epsilon / (4 D_\gY^{p})$, we can conclude that the total complexity combining the three loops of Minimax-AIPE is
\begin{align} \label{eq:main-complexity}
\begin{split}
&\quad 
    \# \text{OuterLoop} \times  \# \text{MiddleLoop} \times  \# \text{InnerLoop} \\
&=\tilde \gO\left( \left( \frac{L_p}{\mu_x} \right)^{\frac{2}{(3p+1)}} \left( \frac{L_p}{\mu_y} \right)^{\frac{2}{(3p+1)}}  \right) = \tilde \gO\left( \left( \frac{L_p D_\gX^p}{\epsilon} \right)^{\frac{2}{(3p+1)}} \left( \frac{L_p D_{\gY}^p}{\epsilon} \right)^{\frac{2}{(3p+1)}}  \right).
\end{split}
\end{align}

\begin{algorithm*}[t]  
\caption{${\rm iProx}_{\Phi} (\bar \vx, \gamma)$ \\ {\color{gray} $\triangleright$ Middle loop: Solve $\min_{\vx \in \gX} \left\{\Psi(\vy; \bar \vx):= \max_{\vy \in \gY} g_\epsilon(\vx,\vy;\bar \vx)\right\}$ with $\tilde \gO((\nicefrac{\gamma}{\mu_y})^{2/(3p+1)})$ steps.} }\label{alg:Minimax-AIPE-mid}
\begin{algorithmic}[1] 
\renewcommand{\algorithmicrequire}{ \textbf{Parameters:}}
\REQUIRE $\gamma, S_{2}, T_{2}, \delta_2$
\FOR{$s=0,\cdots,S_{2}-1$} 
\STATE $\vv_0^{(s)} = \bar \vy_0^{(s)} = \vy^{(s)}_0 = \vy^{(s)}$, $A_0 = 0$, $\lambda_1'^{(s)} = 1$ \\
\FOR{$t = 0,\cdots, T_{2}-1$}  
\STATE \quad Run
OptMS-restart on $g_\epsilon(\,\cdot\,, \vy_t^{(s)}; \bar \vx)$ to obtain
$ \hat \Psi_t^{(s)} = {\rm iFunc}_{\Psi(\,\cdot\,; \bar \vx)}(\vy_t^{(s)},\delta_2)$.
\STATE \quad 
Run OptMS-restart on $g_\epsilon(\,\cdot\,,\tilde \vy_t^{(s)}; \bar \vx)$ to obtain
$\tilde \Psi_t^{(s)} = {\rm iFunc}_{\Psi(\,\cdot\,; \bar \vx)}(\tilde \vy_t^{(s)},\delta_2)$. \\
\quad {\color{gray} $\triangleright$ Both steps require $\tilde \gO((\nicefrac{L_p}{\gamma}+1)^{2/(3p+1)})$ $p$th-order oracles of $f$.}
\STATE \quad Update $a_{t+1}'^{(s)}$, $A_{t+1}'^{(s)}$ as Line \ref{Line:eq-a} in Algorithm \ref{alg:ANPE-restart}.
\STATE \quad $\bar \vy_t^{(s)} = \frac{A_t^{(s)}}{A_{t+1}'^{(s)}} \vy_t^{(s)} + \frac{a_{t+1}'^{(s)}}{A_{t+1}'^{(s)}} \vv_t^{(s)} $  \\
\STATE \quad Invoke Algorithm \ref{alg:Minimax-AIPE-inner} to obtain
$(\tilde \vy_{t+1}^{(s)}, \vu_{t+1}^{(s)})= {{\rm iProx}_{\Psi(\,\cdot\,;\bar \vx)} (\bar \vy_t^{(s)}, \gamma, \delta_2)} $.\\
\quad {\color{gray} $\triangleright$ Requires  $\tilde \gO((\nicefrac{L_p}{\gamma} +1)^{2/(p+1)})$ inner-loop steps.}
\\
\STATE \quad $\lambda_{t+1}^{(s)} = \gamma \Vert \tilde \vy_{t+1}^{(s)} - \bar \vy_t^{(s)} \Vert^{p-1}$ \\
\STATE \quad Update $\gamma_{t+1}^{(s)}$, $a_{t+1}^{(s)}$, $A_{t+1}^{(s)}$, $\lambda_{t+2}'^{(s)}$ as Line \ref{line:update-lmbd} to \ref{line:update-lmbd-end} in Algorithm \ref{alg:ANPE-restart}.
\STATE \quad $ \vy_{t+1}^{(s)} = \frac{(1- \gamma_{t+1}^{(s)}) A_t^{(s)}}{A_{t+1}^{(s)}} \vy_t^{(s)} + \frac{\gamma_{t+1}^{(s)} A_{t+1}'^{(s)}}{A_{t+1}^{(s)}} \tilde \vy_{t+1}^{(s)}$ \\
\STATE \quad Run OptMS-restart
on $g_\epsilon(\,\cdot\,,\tilde \vy_{t+1}^{(s)};\bar \vx)$ to obtain
$\vg_{t+1}^{(s)} = {\rm iGrad}_{\Psi(\cdot;\bar \vx)}(\tilde \vy_{t+1}^{(s)},\delta_2)$.
\\
\quad {\color{gray} $\triangleright$ Requires $\tilde \gO((\nicefrac{L_p}{\gamma} +1)^{2/(3p+1)})$ $p$th-order oracles of $f$.}
\STATE \quad $\vv_{t+1}^{(s)} =  {\rm Proj}_{\gY} (\vv_t^{(s)} - a_{t+1}^{(s)} (\vg_{t+1}^{(s)} + \vu_{t+1}^{(s)}))$ \\
\ENDFOR \\
\STATE Select $\vy^{s+1}$ as the point achieving the lowest value of $\{ \hat \Psi_t^{(s)}(\cdot,\bar \vx),\tilde \Psi_t^{(s)}(\cdot,\bar \vx) \}_{t=0}^{T-1}$. \\
\ENDFOR \\
\STATE Let $\hat \vy = \vy^{(S_2)}$ and run OptMS-restart to find $\hat \vx$ that approximately solves $\min_{\vx \in \gX } g_\epsilon(\vx, \hat \vy)$.
\STATE {\color{gray} $\triangleright$ Take an additional gradient step on $\vx$ to make gradient small.}
\STATE Let $\tilde L_1^x =  L_1 + p (\gamma+\mu_x) D_\gX^{p-1}$ and update $\tilde \vx = {\rm Proj}_\gX(\hat \vx - (1/\tilde L_1^x) \nabla_x g_\epsilon(\hat \vx, \hat \vy; \bar \vx))$ \\
\STATE $
\tilde \vu = \tilde L_1^x(\hat \vx - \tilde \vx) - \nabla_x g_\epsilon(\hat \vx, \hat \vy; \bar \vx)
$.
\RETURN  $(\tilde \vx, \tilde \vu)$
\end{algorithmic}
\end{algorithm*}

For a formal proof, we need to bound all the inexactness in solving sub-problems. Specifically, we show that, given any $\epsilon>0$, there exist constants $\zeta_1,\zeta_2,\zeta_3>0$ corresponding to the precision of Algorithm \ref{alg:Minimax-AIPE}, \ref{alg:Minimax-AIPE-mid}, and \ref{alg:Minimax-AIPE-inner} such that the algorithm can provably output an $\epsilon$-solution to Problem (\ref{prob:main}).
\begin{theorem}[Main Theorem] \label{thm:Minimax-AIPE-CC}
Let the precision $\epsilon>0$ satisfies $\epsilon / \min \{ D_\gX^{p}, D_\gY^p \} \le L_p$.
Under Assumption \ref{asm:D}, \ref{asm:CC}, \ref{asm:grad-lip}, \ref{asm:Hess-lip}, there exist parameters
\begin{align*}
  S_1, S_2, S_3 &= {\rm poly} \log (L_1,D_\gZ, \nicefrac{1}{\epsilon}), \quad \delta_1,\delta_2 = 1 {\big /} {\rm poly} \log (L_1,D_\gZ, \nicefrac{1}{\epsilon}),\\
\gamma,  M  &= \gO(L_p), \quad T_1, T_2 = \gO \left( \left( \nicefrac{L_p}{\epsilon} \right)^{2/(3p+1)} \right) , \quad T_3 = \gO(1),
\end{align*}
such that running Algorithm~\ref{alg:Minimax-AIPE} with the above parameters 
can provably find an $\epsilon$-solution to Problem~(\ref{prob:main}) with $\tilde \gO( (\nicefrac{L_p}{\gamma})^{4/(3p+1)})$  $p$th-order oracle calls following Eq. (\ref{eq:main-complexity}).
\end{theorem}

\begin{proof}
See Appendix \ref{apx:proof-algo} for the proof.

\end{proof}

\begin{algorithm*}[t]  
\caption{${\rm iProx}_{\Psi(\,\cdot\,; \bar \vx)} (\bar \vy, \gamma)$ \\
{\color{gray} $\triangleright$ Inner loop:
Solve $ \min_{\vx \in \gX} \max_{\vy \in \gY} h_\epsilon(\vx,\vy;\bar \vx,\bar \vy)$ with $\tilde \gO((\nicefrac{L_p}{\gamma} + 1)^{2/(p+1)})$ steps.
}
}\label{alg:Minimax-AIPE-inner}
\begin{algorithmic}[1] 
\renewcommand{\algorithmicrequire}{ \textbf{Parameters:}}
\REQUIRE $\gamma, M, S_{3}, T_{3}$
\STATE Define $\vz = (\vx,\vy)$ and $ \mF^{\gamma}(\vz) = 
    (\nabla_x h_\epsilon (\vx,\vy),
    -\nabla_y h_\epsilon (\vx,\vy))^\top
$. 
\STATE Initialize at any $\vz^{(0)} \in \gZ$. 
\FOR{$s=0,\cdots,S_3-1$}
\STATE $\vz^{(s)}_0 = \vz^{(s)}$ 
\STATE {\color{gray}$\triangleright$ Run $p$th-order EG \citep{bullins2022higher} on $h_\epsilon(\vx,\vy; \bar \vx, \bar \vy)$ in each epoch.} 
\FOR{$t=0,\cdots,T_{3}-1$} 
\STATE $\vz_{t+1/2}^{(s)} = \gT_p(\vz_t^{(s)}; \mF^{\gamma}, M)$ {\color{gray} 
\hfill 
$\triangleright$ $p$th-order tensor step.}
\STATE $\eta_t^{(s)} = p! {\big/}  ( M \Vert \vz_{t+1/2}^{(s)} - \vz_t^{(s)} \Vert^{p-1}) $ 
\STATE $\vz_{t+1}^{(s)} = {\rm Proj}_{\gZ} (\vz_t^{(s)} - \eta_t^{(s)} \mF^\gamma (\vz_{t+1/2}^{(s)}))$ {\color{gray} 
\hfill 
$\triangleright$ extragradient step.}
\ENDFOR \vspace{1mm}
\STATE $\vz^{(s+1)} = \sum_{t=0}^{T-1} \eta_t^{(s)} \vz_{t+1/2}^{(s)} {\big/} \sum_{t=0}^{T-1} \eta_t^{(s)}   $ 
\ENDFOR 
\STATE {\color{gray} $\triangleright$ Take an additional gradient step on $\vz$ to make gradient small.} 
\STATE 
$\tilde L_1^\gamma = L_1 + p \max\{(\gamma+\mu_x) D_\gX^{p-1}, (\gamma+\mu_y) D_\gY^{p-1}\}$ 
\STATE $\hat \vz = {\rm Proj}_\gZ(\vz^{(s)} - (1/\tilde L_1^\gamma) \mF^\gamma(\vz^{(s)})$ \\
\STATE $\hat \vc = \tilde L_1^\gamma(\vz^{(s)} - \hat \vz) - \mF^\gamma(\vz^{(s)})
$.
\RETURN $(\hat \vy, \hat \vv)$ {\color{gray} \hfill $\triangleright$ Define $\hat \vz = (\hat \vx, \hat \vy)$ and $\hat \vc = (\hat \vu, \hat \vv)$}
\end{algorithmic}
\end{algorithm*}

\begin{remark} \label{rmk:Catalyst}
In the proof of Theorem \ref{thm:Minimax-AIPE-CC}, we reduce solving the original problem (\ref{prob:main}) to solving $\tilde \gO\left( (\nicefrac{\gamma}{\epsilon})^{4/(3p+1)} \right)$ two-sided regularized minimax problem (\ref{prob:minimax-h}).
Therefore, our framework can be used to accelerate any algorithm $\gM$ that has 
a linear convergence on problem (\ref{prob:minimax-h}). For instance, one can also apply the $p$th-order DE \citep{lin2022perseus} to obtain the same complexity guarantee. Moreover, when $p=2$, we can
also set $\gM$ to be the lazy second-order EG 
\citep{chen2025second}
and adjust the setting of $\gamma$ to obtain an improved computational complexity than Theorem \ref{thm:Minimax-AIPE-CC} by reusing Hessian information. Interested readers can find 
more details in
our conference version~\citep[Theorem~5.7]{chen2025solving}.
\end{remark}

Our $\tilde \gO( (\nicefrac{L_p}{\epsilon})^{4/(3p+1)})$ upper bound in Theorem \ref{thm:Minimax-AIPE-CC} significantly improves the existing $\gO((\nicefrac{L_p}{\epsilon})^{2/(p+1)} )$ result achieved by the $p$th-order EG~\citep{monteiro2012iteration,bullins2022higher,huang2022approximation}.
We conjecture that our upper bound is optimal up to logarithmic factors since it is the multiplication of the optimal $(\nicefrac{L_p}{\epsilon})^{2/(3p+1)}$ complexity \citep{arjevani2019oracle} for minimizing a convex function on $\vx$ and the $(\nicefrac{L_p}{\epsilon})^{2/(3p+1)}$ complexity for maximizing a concave function on $\vy$. However, we can not rigorously prove a matching lower bound in this work.  Instead, we will provide a weaker lower bound in the next section as an initial step towards settling the optimal oracle complexity in minimax optimization.

\section{The $\Omega(\epsilon^{-2/(3p-1)})$ Lower Bound} \label{sec:lower}

Our upper bound in the previous section suggests that existing rates are not optimal. 
In this section, we provide lower bounds for the following algorithm class, which is a natural generalization of the $p$th-order algorithms from convex minimization \citep{nesterov2021inexact,nesterov2023inexact} to convex-concave minimax problems.
The assumption of $\vz_0 = (\vx_0,\vy_0) = (\vzero,\vzero)$ is without loss of generality, since it can always be ensured by changing the function to $f(\vx-  \vx_0, \vy - \vy_0)$.

\begin{dfn}[$p$th-order tensor algorithm] \label{dfn:ZR-alg}
We define the $p$th-order tensor algorithm class for minimax optimization as the set of all algorithm $\gM$ that starts with $\vz_0 = (\vx_0,\vy_0) = (\vzero, \vzero)$, and at each iteration $t$, the algorithm picks $\bar \vx_t \in {\rm Span}(\{ \vx_i \}_{i=0}^{t})$, $\bar \vy_t \in {\rm Span}(\{ \vy_i \}_{i=0}^{t})$, order $q \in \{1,\cdots,p \}$, regularization parameter $M_t$ and then takes one of the following three options:
\begin{enumerate} [label=(\Alph*)]
    \item Generate $\vx_{t+1} = \gA_q(\bar \vx_t; \nabla_x f( \,\cdot \, , \bar \vy_t), M_t  )$ and $\vy_{t+1} = \vy_t$.
    \item Generate $\vy_{t+1} = \gA_q( \bar \vy_t; -\nabla_y f( \bar \vx_t, \,\cdot \,), M_t)$ and $\vx_{t+1} = \vx_t$.
    \item Generate $(\vx_{t+1}, \vy_{t+1}) = \gA_q( \bar \vx_t, \bar \vy_t; \mF, M_t )$.
\end{enumerate}
\end{dfn}

Definition \ref{dfn:ZR-alg} increases the previous algorithm class that always takes option~(C) \citep[Assumption 3.10]{lin2022perseus} by additionally allowing alternative updates in options (A) and (B).
This generalization makes our lower bound hold not only for the previous non-accelerated algorithm based on simultaneous updates \citep{monteiro2012iteration,bullins2022higher,lin2022perseus}, but also hold for our accelerated algorithm but leverages both  simultaneous  and alternating updates.




\subsection{Existing Lower Bounds} \label{sec:diss-lower}

Recent works \citep{adil2022optimal,lin2022perseus} conjectured that the complexity bound of high-order EG can not be improved, which seems to contradict our main result. We clarify the subtleties below by discussing their results.

\citet{lin2022perseus} defined an unscaled hard instance $\bar f: \bar \gX \times \bar \gY \rightarrow \sR$ as 
\begin{align} \label{eq:lin-hard-ins}
\begin{split}
       \bar f(\vx,\vy) =& \frac{L_p}{2^{p+1} (p+1)!} \left( \sum_{i=1}^{4T-1}\left((\vx_{[i]} - \vx_{[i+1]})^p \vy_{[i]} + \left(\vy_{[i]} \right)^{p+1}  \right) + (\vx_{[4T]})^{p} \vy_{[4T]} \right. \\
    & \left. 
    + (\vx_{[4T+1]})^{p} \vy_{[4T+1]}   - \frac{1}{p(p+1)} \sum_{i=2}^{4T} (\vy_{[i]})^{p+1} - \left( \vx_{[1]} -  4T + \frac{1}{p} \right) \vy_{[1]} \right),
\end{split}
\end{align}
where the associated constraint sets $\bar \gX$ and $\bar \gY$ are 
\begin{align} \label{eq:lin-const-set}
\begin{split}
    \bar \gX &= \{ \vx \in \sR^{4T+1} : 0 \le \vx_{[i]} \le 4 T- i+1 ~{\rm and}~ \vx_{[i+1]} \le \vx_{[i]}, ~ \forall i = 1,\cdots, 4T+1 \}, \\
    \bar \gY &= \{ \vy \in \sR^{4T+1} : 0 \le \vy_{[i]} \le 1,~\forall i = 1,\cdots,4 T ~~ {\rm and} ~~ \vy_{[4T+1]} = 0  \}.
\end{split}
\end{align}
Based on this instance, \citet{lin2022perseus} proved a lower bound for $p$th-order algorithms with simultaneous updates. By carefully examining their proof, we find their result also applies to any algorithm in Definition \ref{dfn:ZR-alg} that takes both simultaneous and alternative updates. In summary, we restate their result as follows.



\begin{theorem}[{\citet[Theorem 3.11]{lin2022perseus}}] \label{thm:Lin-LB}
Let $D_{\bar \gX} = 8 T^{3/2}$ and $ D_{\bar \gY} = T^{1/2}$ be the diameters of the constraint sets $\bar \gX$ and $\bar \gY$ defined in Eq. (\ref{eq:lin-const-set}), respectively. The hard instance $\bar f(\vx,\vy)$ defined in Eq. (\ref{eq:lin-hard-ins}) satisfies the following.
\begin{enumerate}
    \item It is convex-concave and has $L_p$-Lipschitz continuous $p$th-order derivatives on~$\bar \gX \times \bar \gY$.
    \item For any $T \in \sN_+$ and any sequence $\{ (\vx_t,\vy_t)\}_{t=0}^{T-1}$ generated by a $p$th-order algorithm in Definition \ref{dfn:ZR-alg}, we have that
    \begin{equation} \label{eq:Gap-Lin}
    {\rm Gap}_{\bar f}(\vx_T,\vy_T) \ge \frac{L_p T}{2^{p-1} (p+1)!} = \Omega \left(\frac{L_p D_{ \bar \gX} D_{\bar \gY}^p}{T^{(p+1)/2}} \right).
\end{equation}
\end{enumerate}
\end{theorem}



At first glance, the above 
$\Omega\left( D_{\bar \gX} D_{\bar \gY}^p / T^{(p+1)/2}\right)$
lower bound 
matches the
$\gO\left( D_{\bar \gZ}^{p+1} / T^{(p+1)/2}\right)$
upper bound achieved by $p$th-order EG \citep[Theorem 4.2]{bullins2022higher} on the $T$ dependency, where $D_{\bar \gZ}$ is the diameter of the set $\bar \gZ = \bar \gX \times \bar \gY$.
However, 
in the above lower bound construction,
the diameters $D_{\bar \gZ}$, $D_{\bar \gX}$, and $D_{\bar \gY}$, implicitly depends on $T$. And
because $D_{\bar \gZ}, D_{\bar \gX} = \Omega(T^{3/2})$ and $D_{\bar \gY} = T^{1/2}$, the upper bound of $p$th-order EG is $\gO(T^{p+1})$, while the lower bound~in Eq. (\ref{eq:Gap-Lin}) is only $\Omega(T)$. 
To give a lower bound for optimizing a convex-concave function $f: \gZ \rightarrow \sR$ such that $\gZ$ has a $T$-independent diameter $D_{\gZ}>0$, we need to rescale the above function and its associated domain as
\begin{equation*}
    f(\vx,\vy) = \bar f(\beta\vx, \beta \vy) / \beta^{p+1}
    , \quad \gZ = \bar \gZ / \beta, \quad {\rm where} ~~ \beta = \frac{\max\{D_{\bar \gX}, D_{\bar \gY} \}}{\sqrt{2} D_\gZ} = \frac{4 \sqrt{2} T^{3/2}}{D_\gZ},
\end{equation*}
where $\bar \gZ / \beta$ is the shorthand for the set $\{ \vz \in \sR^{4T+1} :  \beta \vz \in \bar \gZ \}$.
The above rescaling ensures that $f: \gZ \rightarrow \sR$ still has $L_p$-Lipschitz continuous $p$th-order derivatives on the rescaled domain $\gZ$.
From Eq.~(\ref{eq:Gap-Lin}), we can obtain the following lower bound of 
\begin{equation*}
    {\rm Gap}_{f}(\vx_T,\vy_T) \ge \frac{L_p T}{\beta^{p+1} 2^{p-1} (p+1)!} = \Omega \left(
    \frac{L_p D_\gZ^{p+1}}{T^{(3p+1)/2}} \right),
\end{equation*}
which degenerates to the same $\Omega( \epsilon^{-2/(3p+1)})$ lower bound derived from convex minimization \citep{arjevani2019oracle}. The cause of this degeneration is the asymmetry of the diameters $D_{\bar \gX}$ and $D_{\bar \gY}$ in the construction of Theorem \ref{thm:Lin-LB}, since the rescale factor~$\beta$ has to depend on their maximum rather than the minimum, to remain the same Lipschitz constant of the $p$th-order derivatives.



\subsection{Our Lower Bound} 

In this section, we give a new lower bound construction with symmetric diameters in $\vx$ and $\vy$ to fix the aforementioned issue in \citep{adil2022optimal,lin2022perseus}. 
We construct the unscaled hard instance $f: \bar \gX \times \bar \gY \rightarrow \sR$ as 
\begin{equation} \label{eq:our-hard-ins}
    \bar f(\vx,\vy) =  \frac{L_p}{2^{p+1} p!}\left(\vy_{[1]} (1- \vx_{[1]})^p + \sum_{i=1}^{T} \vy_{[i+1]} (\vx_{[i]} - \vx_{[i+1]})^p \right),
\end{equation}
where the associated constraint sets $\bar \gX$ and $\bar \gY$ are 
\begin{align*}
    \bar \gX &= \{ \vx \in \sR^{T+1} : 0 \le \vx_{[T+1]} \le \vx_{[T]} \le \vx_{[1]} \le 1 \}, \\
    \bar \gY &= \{ \vy \in \sR^{T+1} : 0 \le \vy_{[i]} \le 1,~~ \forall i = 1,\cdots,T+1  \}.
\end{align*}
We summarize the properties of our construction in the following lemma. 
\begin{lemma} \label{lem:pro-hard-ins}
The function $\bar f(\vx,\vy)$ defined in Eq. (\ref{eq:our-hard-ins}) satisfies the following properties.
\begin{enumerate}
    \item It is convex-concave
    and has $L_p$-Lipschitz continuous $p$th-order derivatives on $\bar \gX \times \bar \gY$.
    \item For any algorithm $\gM$ in Definition \ref{dfn:ZR-alg} running on $\bar f(\vx,\vy)$ within $t$ iterations, we have
    \begin{equation} \label{eq:zero-respect}
        {\rm supp}(\vx_{t}) \subseteq \{1, 2,\cdots, t \} \quad \text{and} \quad  {\rm supp}(\vy_{t}) \subseteq \{1, 2,\cdots, t \}.
    \end{equation}
    \item  
    Let $\bar \mF = (\nabla_x \bar f, - \nabla_y \bar f)^\top$. For any $\vz = (\vx,\vy) \in \bar \gX \times \bar \gY$  with $ \vx_{[T+1]} = \vy_{[T+1]} = 0 $, 
    \begin{equation*}
        r_{\bar \mF}^{\rm tan}(\vz) \ge \frac{{\rm Gap}_{\bar f}(\vz)}{D_{\bar \gZ}} \ge   \frac{ D_{\bar \gZ}^{p}}{2^{(p+1)/2} (T+1)^{(3p-1)/2}},
    \end{equation*}
    where $D_{\bar \gZ} = \sqrt{2(T+1)}$ is the diameter of constraint set $\bar \gZ$.
\end{enumerate}
\end{lemma}
\begin{proof}
See Appendix \ref{apx:lb-proof} for the proof.  
\end{proof}
Now, to give a lower bound for optimizing a convex-concave function $f: \gZ \rightarrow \sR$ with $L_p$-Lipschitz continuous $p$th-order derivatives, such that $\gZ$ has a diameter no more than $D_\gZ$, we only need to rescale the above hard instance as well as its domain as 
\begin{equation} \label{eq:our-ins-scale}
    f(\vx,\vy) =  \bar f(\beta \vx, \beta \vy) / \beta^{p+1} , \quad  
    \gZ = \bar \gZ / \beta, \quad 
    {\rm where}~~ \beta = {D_{\bar \gZ}}/{D_{\gZ}}.
\end{equation}
Finally, we obtain the following lower bound.
\begin{theorem} \label{thm:lower}
The hard instance $f(\vx,\vy)$ defined in Eq. (\ref{eq:our-ins-scale}) satisfies the following.
\begin{enumerate}
    \item It is convex-concave and has $L_p$-Lipschitz continuous $p$th-order derivatives on $\gX \times \gY$.
    \item Let $\mF = (\nabla_x  f, - \nabla_y f)^\top$.  For any $T \in \sN_+$ and any sequence $\{ \vz_t\}_{t=0}^{T-1}$ generated by a $p$th-order algorithm $\gM$ in Definition \ref{dfn:ZR-alg}, we have that
     \begin{equation} \label{eq:our-lb}
     r^{\rm tan}_{\mF}(\vz_T) = \Omega \left(  \frac{L_p D_{\gZ}^{p}}{ (T+1)^{(3p-1)/2}} \right).
 \end{equation}
 In other words, the number of $p$th-order oracle calls for the algorithm class in Definition~\ref{dfn:ZR-alg} to find an $\epsilon$-solution is lower bounded by $\Omega( (\nicefrac{D_{\gZ}^p L_p}{\epsilon} )^{2/(3p-1)}) $.
\end{enumerate}
\end{theorem}

\begin{proof}
First, item 1 directly follows from Lemma \ref{lem:pro-hard-ins} (item 1) and the scaling in 
Eq. (\ref{eq:our-ins-scale}). Second, from  Lemma~\ref{lem:pro-hard-ins} (item 2), we know that the sequence $\{ (\vx_t,\vy_t)\}_{t=0}^T$ generated by $\gM$ satisfies $(\vx_T)_{[T+1]} = (\vy_T)_{[T+1]} = 0 $, and thus Eq. (\ref{eq:our-lb}) follows from Lemma \ref{lem:pro-hard-ins} (item~3).  
\end{proof}
The theorem shows an $\Omega (\epsilon^{-2/(3p-1)})$ lower bound holds for all $p$th-order algorithms in Definition~\ref{dfn:ZR-alg}.
Note that our hard instance also has $\gO(1)$-Lipschitz continuous $q$th-order derivatives for all $q = 1,\cdots,p$, which means the 
lower-order Lipschitzness does not change the lower bound.

When $p=1$, our lower bound is tight as the $\gO(\epsilon^{-1})$ upper bound can be achieved by the EG method \citep{korpelevich1976extragradient}. When $p \ge 2$, our lower bound still has a gap between the $\tilde \gO(\epsilon^{-4/(3p+1)})$ upper bound established in this work. 
We conjecture that our current lower bound is possibly not tight, but improving it looks challenging as we find it difficult to come up with a convex-concave function with $p$th-order derivatives besides the function $x^p y$ we used in Eq. (\ref{eq:our-hard-ins}). 
We hope our result can be an initial step towards proving a tighter lower bound in the future.

\section{Conclusion and Future Works}

In this paper, we propose a $p$th-order algorithm Minimax-AIPE for convex-concave minimax problems. Our theoretical result shows our proposed method achieves an $\tilde \gO(\epsilon^{-4/(3p+1)})$  upper complexity bound and thus improves the best-known $\gO(\epsilon^{-2/(p+1)})$ result. 
We also establish an $\Omega(\epsilon^{-2/(3p-1)})$ lower bound for all $p$th-order algorithms.

One important future direction is to further close the gap between our upper and lower bounds. Other possible directions include shaving the logarithmic factor in the upper bound by extending the existing technique from $p=1$
\citep{kovalev2022firstSC,carmon2022recapp} to all $p \in \sN_+$, and using our framework to accelerate other methods such as stochastic algorithms \citep{wang2019stochastic,zhou2019stochastic,tripuraneni2018stochastic} and quasi-Newton algorithms~\citep{jiang2023online}. 

\section*{Acknowledgments} Lesi Chen thanks Weiqiang Zheng for discussions on reference~\citep{cai2022finite} and Tianyi Lin for discussions on reference~\citep{lin2022perseus}.

\bibliographystyle{plainnat}
\bibliography{sample}
\appendix

\section{Proof of Theorem \ref{thm:AIPE}} \label{apx:proof-AIPE}

\begin{proof}

Let ${\rm Gap}_h(\vz) = h(\vz) - h(\vz^*)$. We prove that the parameter setting $T = \gO\left((\gamma/ \mu)^{2/(3p+1)}\right)$ 
ensures that, for all $s = 0,\cdots,S-1$, each epoch of the algorithm satisfies
\begin{equation} \label{eq:epoch-half}
{\rm Gap}_h(\vz^{(s+1)}) \le \frac{1}{2} {\rm Gap}_h(\vz^{(s)}) \le  \cdots \le \frac{1}{2^{s+1}} \Delta= \epsilon^{(s+1)},
\end{equation}
where $\epsilon^{(s)}:= \Delta / 2^s$ is the accuracy for $\vz^{(s)}$ to achieve. If we can prove Eq. (\ref{eq:epoch-half}), then it implies that setting $S = \left \lceil \log_2 ( \Delta / \epsilon) \right \rceil$ suffices to find ${\rm Gap}_h(\vz^S) \le \epsilon$.
Below, we analyze epoch $s$ to show that Eq. (\ref{eq:epoch-half}) holds.
Following the notations in \citep{carmon2022optimal}, we define $E_t^{(s)}:= {\rm Gap}(\vz_t^{(s)})$, $R_t^{(s)}:=  \Vert \vv_t^{(s)} - \vz^* \Vert^2/2$ and $N_{t+1}^{(s)}:=\Vert \tilde \vz_{t+1}^{(s)} - \bar \vz_t^{(s)} \Vert^2/2$.
By the update rule of $\vv_t^{(s)}$, we have that
\begin{align} \label{eq:MS-1}
\begin{split}
    R_{t+1}^{(s)} &=\frac{1}{2} \left \Vert {\rm Proj}_{\gZ} (\vv_t^{(s)} - a_{t+1}^{(s)}  (\vg_{t+1}^{(s)} + \vu_{t+1}^{(s)})) - \vz^* \right \Vert^2 \\
    &\le \frac{1}{2} \left \Vert (\vv_t^{(s)} - a_{t+1}^{(s)} (\vg_{t+1}^{(s)} + \vu_{t+1}^{(s)})) - \vz^* \right \Vert^2 \\
    &= R_t^{(s)} + a_{t+1}^{(s)} \langle \vg_{t+1}^{(s)} + \vu_{t+1}^{(s)}, \vz^* - \vv_t^{(s)} \rangle + \frac{(a_{t+1}^{(s)})^2}{2} \Vert \vg_{t+1}^{(s)} + \vu_{t+1}^{(s)} \Vert^2 \\
    &\le  R_t^{(s)} + a_{t+1}^{(s)} \langle  \nabla h(\tilde \vz_{t+1}^{(s)}) + \vu_{t+1}^{(s)}, \vz^* - \vv_t^{(s)} \rangle + a_{t+1}^{(s)} \delta \Vert \vv_t^{(s)} - \vz^* \Vert \\
    &\quad + (a_{t+1}^{(s)})^2 \Vert \nabla h(\tilde \vz_{t+1}^{(s)}) + \vu_{t+1}^{(s)} \Vert^2   + (a_{t+1}^{(s)})^2 \delta^2,
\end{split}
\end{align}
where the first inequality uses the non-expansiveness of the projection operator, and the second one uses both the triangle inequality and Young's inequality.
Next, by the update rule of $\bar \vz_t^{(s)} = \frac{A_t^{(s)}}{A_{t+1}'^{(s)}} \vz_t^{(s)} + \frac{a_{t+1}'^{(s)}}{A_{t+1}'^{(s)}} \vv_t^{(s)} $ and $A_{t+1}'^{(s)} = A_t^{(s)} + a_{t+1}'^{(s)}$, we  have
\begin{equation*}
    a_{t+1}'^{(s)} \vv_t^{(s)}= A_{t+1}'^{(s)} \bar \vz_t^{(s)}  - A_t^{(s)} \vz_t^{(s)} = a_{t+1}'^{(s)} \tilde \vz_{t+1}^{(s)} - A_t^{(s)} (\vz_t^{(s)} - \tilde \vz_{t+1}^{(s)}) + A_{t+1}'^{(s)} (\bar \vz_t^{(s)} - \tilde \vz_{t+1}^{(s)}) .
\end{equation*}
Then subtracting $a_{t+1}'^{(s)} \vz^*$ and taking inner product with $ \nabla h(\tilde \vz_{t+1}^{(s)}) + \vu_{t+1}^{(s)} $ yields
\begin{align}\label{eq:to-plug}
\begin{split} 
    &\quad a_{t+1}'^{(s)} \langle \nabla h(\tilde \vz_{t+1}^{(s)}) + \vu_{t+1}^{(s)}, \vz^* - \vv_t^{(s)} \rangle \\
&= \langle \nabla h(\tilde \vz_{t+1}^{(s)}) + \vu_{t+1}^{(s)}, a_{t+1}'^{(s)}  ( \vz^*-  \tilde \vz_{t+1}^{(s)}) + A_t^{(s)} (\vz_t^{(s)} - \tilde \vz_{t+1}^{(s)})  + A_{t+1}'^{(s)} (\tilde \vz_{t+1}^{(s)}- \bar \vz_t^{(s)}) \rangle  \\
&\le a_{t+1}'^{(s)} ( h(\vz^*) - h(\tilde \vz_{t+1}^{(s)})  )  + A_t^{(s)} (h(\vz_t^{(s)}) - h(\tilde \vz_{t+1}^{(s)})) + A_{t+1}'^{(s)} \langle \nabla h(\tilde \vz_{t+1}^{(s)}) + \vu_{t+1}^{(s)}, \tilde \vz_{t+1}^{(s)} - \bar \vz_t^{(s)} \rangle \\
 &\le a_{t+1}'^{(s)} ( h(\vz^*) - h(\tilde \vz_{t+1}^{(s)})  )  + A_t^{(s)} (h(\vz_t^{(s)}) - h(\tilde \vz_{t+1}^{(s)})) - \frac{A_{t+1}'^{(s)}}{2 \lambda_{t+1}^{(s)}} \Vert \nabla h(\tilde \vz_{t+1}^{(s)}) + \vu_{t+1}^{(s)} \Vert^2 \\
    &\quad + \frac{A_{t+1}'^{(s)}}{2 \lambda_{t+1}^{(s)}} \Vert \nabla h(\tilde \vz_{t+1}^{(s)}) + \vu_{t+1}^{(s)} + \lambda_{t+1}^{(s)} (\tilde \vz_{t+1}^{(s)} - \bar \vz_t^{(s)}) \Vert^2  - \frac{A_{t+1}'^{(s)} \lambda_{t+1}^{(s)}}{2} \Vert \tilde \vz_{t+1}^{(s)} - \bar \vz_t^{(s)} \Vert^2,
\end{split}
\end{align}
where the last step uses the convexity of $h$ and $\vu_{t+1}^{(s)} \in \gN_{\gZ}(\tilde \vz_{t+1}^{(s)})$ to upper-bound the first two terms and upper-bounds the inner product in the last term
using the following identity. 
\begin{align*}
\langle \nabla h(\tilde \vz_{t+1}^{(s)}) + \vu_{t+1}^{(s)}, \tilde \vz_{t+1}^{(s)} - \bar \vz_t^{(s)} \rangle 
     &=\frac{1}{2 \lambda_{t+1}^{(s)}} \Vert \nabla h(\tilde \vz_{t+1}^{(s)}) + \vu_{t+1}^{(s)} + \lambda_{t+1}^{(s)} (\tilde \vz_{t+1}^{(s)} - \bar \vz_t^{(s)}) \Vert^2 \\
     &\quad - \frac{1}{2 \lambda_{t+1}^{(s)}} \Vert \nabla h(\tilde \vz_{t+1}^{(s)}) + \vu_{t+1}^{(s)} \Vert^2  - \frac{\lambda_{t+1}^{(s)}}{2} \Vert \tilde \vz_{t+1}^{(s)} - \bar \vz_t^{(s)} \Vert^2.
\end{align*}
Then, using the update of $\tilde \vz_{t+1}^{(s)}, \vu_{t+1}^{(s)} = {\rm iProx}_h (\bar \vz_t^{(s)}, \gamma, \delta)$, we have 
\begin{align} \label{eq:MS-2}
\begin{split}
 a_{t+1}'^{(s)} \langle \nabla h(\tilde \vz_{t+1}^{(s)}) + \vu_{t+1}^{(s)}, \vz^* - \vv_t^{(s)} \rangle \le& a_{t+1}'^{(s)} ( h(\vz^*) - h(\tilde \vz_{t+1}^{(s)})  )  + A_t^{(s)} (h(\vz_t^{(s)}) - h(\tilde \vz_{t+1}^{(s)})) \\
    & - \frac{A_{t+1}'^{(s)}}{2 \lambda_{t+1}^{(s)} } \Vert \nabla h(\tilde \vz_{t+1}^{(s)}) + \vu_{t+1}^{(s)} \Vert^2 - \frac{1}{2} \lambda_{t+1}^{(s)} A_{t+1}'^{(s)}  N_{t+1}^{(s)} + \frac{A_{t+1}'^{(s)}}{\lambda_{t+1}^{(s)}} \delta^2,
\end{split}
\end{align}
where the last step uses the approximate proximal condition (\ref{eq:cond-appr-prox}) and the Young's inequality.
Now we separately consider the two cases $\lambda_{t+1}^{(s)} \le \lambda_{t+1}'^{(s)}$ (\textit{Case I}) and $\lambda_{t+1}^{(s)} > \lambda_{t+1}'^{(s)}$ (\textit{Case II}). In the first case, we have that $\vz_{t+1}^{(s)} = \tilde \vz_{t+1}^{(s)} $, $a_{t+1}^{(s)} = a_{t+1}'^{(s)}$, $A_{t+1}^{(s)} = A_{t+1}'^{(s)}$ and $A_{t+1}^{(s)} = 2 \lambda_{t+1}'^{(s)} (a_{t+1}^{(s)})^2$. Then we can directly sum up Eq. (\ref{eq:MS-1}) and Eq. (\ref{eq:MS-2}) to obtain that
\[
\textit{Case I:} \ A_{t+1}^{(s)} E_{t+1}^{(s)} + R_{t+1}^{(s)} + \frac{1}{2} \lambda_{t+1}^{(s)} A_{t+1}^{(s)} N_{t+1}^{(s)}  \le A_t^{(s)} E_t^{(s)} + R_t^{(s)} + a_{t+1}^{(s)} \delta \Vert \vv_t^{(s)} - \vz^* \Vert + 3 (a_{t+1}^{(s)})^2 \delta^2.
\]
In the second case, we have $A_{t+1}^{(s)} = (1-\gamma_{t+1}^{(s)}) A_t^{(s)} + \gamma_{t+1}^{(s)} A_{t+1}'^{(s)}$, where $\gamma_{t+1}^{(s)}= \lambda_{t+1}'^{(s)} / \lambda_{t+1}^{(s)}$. Therefore, by the convexity of $h$, we have 
\begin{equation*}
    A_{t+1}^{(s)} E_{t+1}^{(s)} \le (1- \gamma_{t+1}^{(s)}) A_t^{(s)} E_t^{(s)} + \gamma_{t+1}^{(s)} A_{t+1}'^{(s)} (h(\tilde \vz_{t+1}^{(s)}) - h(\vz^*)).
\end{equation*}
We use Eq. (\ref{eq:MS-2}) multiplied by $\gamma_{t+1}^{(s)}$ to give an upper bound of the above inequality as
\begin{align} \label{eq:MS-3}
\begin{split}
       A_{t+1}^{(s)} E_{t+1}^{(s)} &\le A_t^{(s)} E_t^{(s)} + a_{t+1}^{(s)} \langle \nabla h(\tilde \vz_{t+1}^{(s)}) + \vu_{t+1}^{(s)}, \vz^* - \vv_t^{(s)} \rangle \\
    &- \frac{1}{2} \lambda_{t+1}'^{(s)} A_{t+1}'^{(s)} N_{t+1}^{(s)} - \frac{\gamma_{t+1}^{(s)} A_{t+1}'^{(s)}}{2 \lambda_{t+1}^{(s)}} \Vert \nabla h(\tilde \vz_{t+1}^{(s)}) + \vu_{t+1}^{(s)} \Vert^2 + \frac{\gamma_{t+1}^{(s)} A_{t+1}'^{(s)}}{\lambda_{t+1}^{(s)}} \delta^2, 
\end{split}
\end{align}
where we also use the facts $ \gamma_{t+1}^{(s)} a_{t+1}'^{(s)} = a_{t+1}^{(s)}$ and $\lambda_{t+1}'^{(s)} = \gamma_{t+1}^{(s)} \lambda_{t+1}^{(s)}$. Note that
\begin{equation*}
    \frac{\gamma_{t+1}^{(s)} A_{t+1}'^{(s)}}{\lambda_{t+1}^{(s)}} = \frac{(\gamma_{t+1}^{(s)})^2 A_{t+1}'^2 }{\lambda_{t+1}'^{(s)}} = 2 (\gamma_{t+1}^{(s)} a_{t+1}'^{(s)})^2 = 2(a_{t+1}^{(s)})^2.
\end{equation*}
Summing up Eq. (\ref{eq:MS-1}) and Eq. (\ref{eq:MS-3}) yields that
\[
    \textit{Case II:} \ A_{t+1}^{(s)} E_{t+1}^{(s)} + R_{t+1}^{(s)} + \frac{1}{2} \lambda_{t+1}'^{(s)} A_{t+1}'^{(s)} N_{t+1}^{(s)}\le A_t^{(s)} E_t^{(s)} + R_t^{(s)} + a_{t+1}^{(s)} \delta \Vert \vv_t^{(s)} - \vz^* \Vert + 3 (a_{t+1}^{(s)})^2 \delta^2.
\]
The inequality for both \textit{Case I} and \textit{Case II} can be unified as
\begin{equation} \label{eq:MS-unify}
    A_{t+1}^{(s)} E_{t+1}^{(s)} +R_{t+1}^{(s)} + \frac{1}{2} \min\{ \lambda_{t+1}^{(s)}, \lambda_{t+1}'^{(s)}\} A_{t+1}'^{(s)} N_{t+1}^{(s)} \le A_t^{(s)} E_t^{(s)} + R_t^{(s)} + \delta_{t+1}^{(s)},
    \end{equation}
where $\delta_{t+1}^{(s)}:=  a_{t+1}^{(s)} \delta \Vert \vv_t^{(s)} - \vz^* \Vert + 3 (a_{t+1}^{(s)})^2 \delta^2$. 
We let $\delta$ be sufficiently small such that 
\begin{equation} \label{eq:dfn-beta-s}
 \beta^{(s)}:= \min \left\{ D_\gZ, \left( \frac{(p+1) \epsilon^{(s)}}{\mu} \right)^{1/(p+1)} \right\} \ge A_{T}^{(s)}  \delta /c , \quad \forall s = 0,\cdots,S-1 
\end{equation}
for a constant $c>0$ that we specify later, and 
$\beta^{(s)}$ is an upper bound of $\Vert \vz^{(s)} - \vz^* \Vert$ by Lemma~\ref{lem:UC-grad-dominant} since $h: \sR^d \rightarrow \sR$ is $(p+1)$th-order $\mu$-uniformly convex.
Let $K= \arg \max_{1 \le j \le T} \Vert \vv_j^{(s)} - \vz^* \Vert$.
Telescoping $\delta_{t+1}^{(s)} = a_{t+1}^{(s)} \delta \Vert \vv_t^{(s)} - \vz^* \Vert + 3 (a_{t+1}^{(s)})^2 \delta^2$ for all $t = 0,\cdots,T-1$ yields 
\begin{equation} \label{eq:sum-delta}
\sum_{t=0}^{T-1} \delta_{t+1}^{(s)} \le A_{T}^{(s)} \delta \Vert v_{K}^{(s)} - \vz^* \Vert + 3 (A_{T}^{(s)})^2 \delta^2  \le c \beta^{(s)} \Vert v_{K}^{(s)} - \vz^* \Vert + 3 \left(c \beta^{(s)} \right)^2,
\end{equation}
where the first inequality uses $\sum_{t=0}^{T-1} (a_{t+1}^{(s)})^2 \le \left( \sum_{t=0}^{T-1} a_{t+1}^{(s)} \right)^2 = \left( A_T^{(s)} \right)^2$ and the 
second 
one uses Eq.~(\ref{eq:dfn-beta-s}) that $A_T^{(s)} \delta \le c \beta^{(s)}$. In addition, telescoping Eq.~(\ref{eq:MS-unify}) over $t = 0,\cdots,K-1$ yields
\begin{equation} \label{eq:recur-Rk}
        R_K^{(s)} \le \frac{1}{2} \beta^{(s)} + \sum_{t=0}^{K-1} \delta_{t+1}^{(s)} \le \frac{1}{2} \beta^{(s)} + c \beta^{(s)} \max_{0 \le t \le K-1} \Vert \vv_t^{(s)} - \vz^* \Vert + 3 c^2 (\beta^{(s)})^2,
\end{equation}
where the first inequality uses
$A_0^{(s)} = 0$ and $R_0^{(s)} = \frac{1}{2} \Vert \vz^{(s)} - \vz^* \Vert^2 \le \frac{1}{2} \left(\beta^{(s)} \right)^2$, and the second one uses Eq.~(\ref{eq:sum-delta}).
Since $R_t^{(s)} = \frac{1}{2} \Vert \vv_t^{(s)} - \vz^* \Vert^2$,  Eq. (\ref{eq:recur-Rk}) implies that 
\begin{equation*}
    \frac{1}{2} \Vert \vv_{K}^{(s)} - \vz^* \Vert^2 \le \left( \frac{1}{2} +  3 c^2 \right) \left(\beta^{(s)} \right)^2 + c \beta^{(s)} \Vert \vv_{K}^{(s)} - \vz^* \Vert.
\end{equation*}
Solving the above quadratic inequality with respect to the variable $\Vert \vv_{K}^{(s)} - \vz^* \Vert$ yields
\begin{equation} \label{eq:ub-Rk}
\Vert \vv_{K}^{(s)} - \vz^* \Vert \le \left(c+ \sqrt{1 + 7 c^2} \right) \beta^{(s)} \le (1+5c) \beta^{(s)}.
\end{equation}
Then, combining Eq. (\ref{eq:sum-delta}) and (\ref{eq:ub-Rk}) yields
\begin{equation} \label{eq:ub-sum-delta}
\sum_{t=0}^{T-1} \delta_{t+1}^{(s)} \le c(1+8 c) (\beta^{(s)})^2
\end{equation}
Now, we telescope Eq. (\ref{eq:MS-unify}) over $t=0,1,\cdots,T-1$ and substitute Eq. (\ref{eq:ub-sum-delta}) to get
\begin{equation} \label{eq:our-MS}
    \ A_{T}^{(s)} E_{T}^{(s)}+ R_{T}^{(s)} + \frac{1}{2} \sum_{t=0}^{T-1} \min\{ \lambda_{t+1}^{(s)}, \lambda_{t+1}'^{(s)}\} A_{t+1}'^{(s)} N_{t+1}^{(s)}\le \frac{1}{2} \left(\beta^{(s)}\right)^2 + c(1+8 c) \left(\beta^{(s)} \right)^2 \le \left(\beta^{(s)} \right)^2,
\end{equation}
where we let the last inequality hold by setting $c \le (1 + \sqrt{17}) / 16$. 
Since Eq. (\ref{eq:our-MS}) implies 
\citep[ Eq. (9)]{carmon2022optimal} with their $R_0^{(s)}$ now being replaced by its upper bound $\beta^{(s)}$. We can then follow the remaining steps as the proof of \citep[Theorem 1]{carmon2022optimal}
to show that the algorithm finds $E_t^{(s)} \le \epsilon^{(s+1)}$ in no more than
\begin{equation*}
T = \gO\left( \left( \frac{\gamma (\beta^{(s)})^{p+1}}{\epsilon^{(s+1)}} \right)^{2/ (3p+1)}  \right) = \gO\left( (\gamma / \mu)^{2/(3p+1)} \right)
\end{equation*}
iterations, where the last equality plugs in $\beta^{(s)}$ in Eq. (\ref{eq:dfn-beta-s}) and $  \epsilon^{(s+1)} = \epsilon^{(s)}/2$ in Eq. (\ref{eq:epoch-half}).
Note that we can assume without loss of generality that in epoch $s$ we have ${\rm Gap}_h(\vz) \ge \epsilon^{(s+1)}/2$ for all $\vz \in \{\vz_t^{(s)}, \tilde \vz_t^{(s)} \}$ and $t = 0,\cdots,T-1$. Otherwise, we know 
that Line~\ref{line:select-best-beg} and Line~\ref{line:select-best-end} can guarantee to output a point satisfying ${\rm Gap}_h(\vz^{(s+1)}) \le \epsilon^{(s+1)}$ as required by Eq. (\ref{eq:epoch-half}) under the setting that $\delta \le \epsilon^{(s+1)} /2$. Following the same analysis with $s = 0$, we can also assume that $\Delta \ge \epsilon/2$.

The last thing to show is how to fulfill the goal of $\delta \le c {\beta^{(s)}}/{A_{T}^{(s)}}$ in Eq. (\ref{eq:dfn-beta-s}) for all $s = 0,\cdots,S-1$. By the definition of $\beta^{(s)}$ in Eq. (\ref{eq:dfn-beta-s}) and $\epsilon^{(s)} = \Delta / 2^s \ge \epsilon / 2^{s+1}$, we can lower-bound $\beta^{(s)}$ by $\underline{C}$, where
\begin{equation} \label{eq:lb-beta-s}
    \underline{C} := 
    \min \left\{D_\gZ, \left( \frac{(p+1) \epsilon}{\mu 2^{S}} \right)^{1/(p+1)}
    \right\} \le \beta^{(s)}, \quad \forall s = 0,\cdots,S-1.
\end{equation}
Therefore, we only need to give an upper bound of $A_T^{(s)}$, which is proved by induction as follows. Assume that we have found a constant $\bar C>0$ such that $A_t^{(s)} \le \bar C$ for all $t \le K$, we need to prove that the same constant also fulfills $A_{t}^{(s)} \le \bar C$ for $t ={K+1}$ under the setting $\delta \le c \underline{C} / \bar C$. Since Eq.~(\ref{eq:our-MS}) with $T=K$ holds under the induction hypothesis, we have that $\epsilon^{(s+1)}/2 < E_{K}^{(s)} \le (\beta^{(s)})^2/ A_K^{(s)}$, meaning that $A_K^{(s)} \le 2(\beta^{(s)})^2/ \epsilon^{(s+1)} \le 2 D_\gZ^2 / \epsilon^{(s+1)}$. Then from the update  $A_{t+1}'^{(s)} = 2 \lambda_{t+1}'^{(s)} \left(a_{t+1}'^{(s)}\right)^2$ we know that
\begin{equation} \label{eq:upper-at}
    a_{K+1}'^{(s)} = \frac{1 + \sqrt{1 + 8 \lambda_{K+1}'^{(s)} A_K^{(s)}}}{4 \lambda_{K+1}'^{(s)}} \le \frac{1}{2 \lambda_{K+1}'^{(s)}} + \frac{1}{2} \sqrt{\frac{2A_K^{(s)}}{\lambda_{K+1}'^{(s)}}} < \frac{1}{2 \lambda_{K+1}'^{(s)}} + D_\gZ \sqrt{\frac{1}{\epsilon^{(s+1)} \lambda_{K+1}'^{(s)}}}.
\end{equation}
Now, we need a lower bound of $\lambda_{K+1}'^{(s)}$ to upper bound the above quantity. We directly show a stronger result bound of the lower bound of all $\lambda_{t}'^{(s)}$.
As an intermediate goal, we first analyze~$\lambda_{t}^{(s)}$. 
Using the update of $\tilde \vz_{t}^{(s)}, \vu_{t}^{(s)} = {\rm iProx}_h (\bar \vz_{t-1}^{(s)}, \gamma, \delta)$
and the approximate proximal point condition (\ref{eq:cond-appr-prox}), we have
\begin{equation*}
    \frac{3 \lambda_{t}^{(s)}\Vert \tilde \vz_{t}^{(s)} - \bar \vz_{t-1}^{(s)} \Vert}{2} \ge  \Vert \nabla h(\tilde \vz_{t}^{(s)}) + \vu_{t}^{(s)} \Vert  - \delta \ge \frac{{\rm Gap}_h(\tilde \vz_t^{(s)})}{\Vert \tilde \vz_{t}^{(s)} - \vz^* \Vert} - \delta, \quad \forall t  =1,\cdots, T,
\end{equation*}
where the last step uses ${\rm Gap}_h (\tilde \vz_t^{(s)}) \le \langle \nabla h(\tilde \vz_t^{(s)}) + \vu_t^{(s)}, \tilde \vz_t^{(s)} - \vz^* \rangle \le \Vert \nabla h(\tilde \vz_t^{(s)}) + \vu_t^{(s)} \Vert \Vert \tilde \vz_t^{(s)} - \vz^* \Vert$ implied by the convexity of $h$.
Since ${\rm Gap}_h(\tilde \vz_t^{(s)}) \ge \epsilon^{(s+1)} / 2$, setting $\delta \le \epsilon^{(s+1)}/ (4 D_\gZ)$ ensures that 
\begin{equation} \label{eq:lambda-t-prime}
    \lambda_{t}^{(s)} \ge \frac{3}{2 \Vert \tilde \vz_{t}^{(s)} - \bar \vz_{t-1}^{(s)} \Vert} \left( 
    \frac{{\rm Gap}_h(\tilde \vz_t^{(s)})}{\Vert \tilde \vz_{t}^{(s)} - \vz^* \Vert} - \delta 
    \right) \ge  \frac{\epsilon^{(s+1)}}{6 D_\gZ^2}, \quad \forall t = 1,\cdots,T.
\end{equation}
Let $k$ the the last iterate before $K+1$ such that $ \lambda_k'^{(s)} > \lambda_k^{(s)}$, \textit{i.e.}, $k = \arg \max_{1 \le k' \le K} \{ \lambda_{k'}'^{(s)} > \lambda_{k'}^{(s)}\}$.
If $k=K$, then Eq. (\ref{eq:lambda-t-prime}) already gives a lower bound of $\lambda_{K+1}'^{(s)}$ since $\lambda_{K+1}'^{(s)} \ge \lambda_K'^{(s)}/2 > \lambda_K^{(s)}$.
Otherwise, if $k <K$, we know from the update rule of $\lambda_t'^{(s)}$ that 
\begin{equation} \label{eq:lambda-prime-K-1}
   \lambda_{K+1}'^{(s)} \ge \lambda_{K}'^{(s)}/2 =  2^{K-k-2} \lambda_{k+1}'^{(s)} = 2^{K-k-3} \lambda_k'^{(s)} > \lambda_k^{(s)} /4, \quad \forall t = 1,\cdots,T.
\end{equation}
Combining Eq. (\ref{eq:upper-at}), (\ref{eq:lambda-t-prime}), and (\ref{eq:lambda-prime-K-1}), we have
\begin{equation} \label{eq:ub-aK1}
    a_{K+1}' \le \frac{(12 + 2 \sqrt{6}) D_\gZ^2}{ \epsilon^{(s+1)}}.  
\end{equation}
Therefore, we can define the following constant $\bar C$ as an upper bound of $A_{K+1}^{(s)}$:
\begin{align} \label{eq:ub-AT}
\begin{split}
\bar C &:= \frac{2^{S+2} (7 +\sqrt{6}) D_\gZ^2}{\epsilon} \ge \frac{(14 + 2 \sqrt{6}) D_\gZ^2}{ \epsilon^{(s+1)}}  \\
&\ge \frac{2 D_\gZ^2}{\epsilon^{(s+1)}} + a_{K+1}'^{(s)} \ge  A_{K}^{(s)} + a_{K+1}'^{(s)} \ge A_{K}^{(s)} + a_{K+1}^{(s)}  = A_{K+1}^{(s)}, \quad \forall s = 0,\cdots,S-1,    
\end{split}
\end{align}
which finishes the induction. In the above, the first inequality uses the definition of $\epsilon^{(s+1)}$, the second uses Eq. (\ref{eq:ub-aK1}), and the third uses $ 2 D_\gZ^2 / \epsilon^{(s+1)} \ge A_K^{(s)}$.
And the induction base trivially holds since $A_0^{(s)} = 0$.
Note that $\bar C$ can be determined before running the algorithm. We can set $\delta \le c \underline{C} / \bar C$ and run the algorithm, and the above induction shows that $A_T^{(s)} \le \bar C$ always holds for all $s = 0,\cdots,S-1$. 
From the previous analysis, we know that the convergence indicated by Eq. (\ref{eq:our-MS}) holds, leading to the claimed complexity.  Finally,  we get Eq. (\ref{eq:cond-delta}) by combining all the conditions  $\delta \le c \underline{C} / \bar C$, $\delta \le \epsilon^{(s+1)} / (4 D_\gZ)$, $\delta \le \epsilon^{(s+1)} / 2$, the settings of $S = \lceil \log_2 (\nicefrac{\Delta}{\epsilon}) \rceil \le 1 + \log_2 (\nicefrac{\Delta}{\epsilon})$, $c = (1+ \sqrt{17}) / 16$, and $\underline{C}$, $\bar C$ in Eq. (\ref{eq:lb-beta-s}), (\ref{eq:ub-AT}).

\end{proof}

\section{Proof of Theorem \ref{thm:Minimax-AIPE-CC}} \label{apx:proof-algo}

Recall 
the primal objective $\Phi(\vx) := \max_{\vy \in \gY} f_\epsilon(\vx,\vy)$ and dual objective $\Psi(\vy; \bar \vx):= \min_{\vx \in \gX} g_\epsilon(\vx,\vy ; \bar \vx)$, where $g_\epsilon(\vx,\vy; \bar \vx) := f_\epsilon(\vx,\vy) + (\gamma / (p+1)) \Vert \vx- \bar \vx \Vert^{p+1}$. 
Combining Lemma \ref{lem:cubic-func} and Lemma \ref{lem:benigh-reg-f} (item 1), 
we can easily obtain the following properties for the objectives $f_{\epsilon}(\vx,\vy)$,   $g_{\epsilon}(\vx,\vy; \bar \vx)$,  $h_{\epsilon}(\vx,\vy; \bar \vx,\bar \vy)$, $\Phi(\vx)$, and $\Psi(\vy; \bar \vx)$.
\begin{lemma} \label{lem:prop-funcs}
Under Assumption \ref{asm:D}, \ref{asm:CC}, \ref{asm:grad-lip}, and \ref{asm:Hess-lip}, we have the following:
\begin{enumerate}
    \item $f_{\epsilon}(\vx,\vy)$ is $(p+1)$th-order $(\mu_x/2^{p-1} )$-uniformly-convex-$(\mu_y/2^{p-1})$-uniformly-concave; has $(L_1 + p \mu_x D_\gX^{p-1})$-Lipschitz continuous gradients and $(L_1+ p! \mu_x)$-Lipschitz continuous $p$th-order in $\vx$, $(L_1 + p \mu_y D_\gY^{p-1})$-Lipschitz continuous gradients and $(L_p+ p!\mu_y)$-Lipschitz continuous $p$th-order derivatives in $\vy$.
    \item $g_{\epsilon}(\vx,\vy; \bar \vx)$ is $(p+1)$th-order $((\gamma +\mu_x)/2^{p-1} )$-uniformly-convex-$(\mu_y/2^{p-1})$-uniformly-concave; has $(L_1+ p(\gamma+\mu_x) D_\gX^{p-1})$-Lipschitz continuous gradients and 
    $(L_1+ p!(\gamma +\mu_x))$-Lipschitz continuous $p$th-order derivatives
    in $\vx$, $(L_1 + p \mu_y D_\gY^{p-1})$-Lipschitz continuous gradients and $(L_p + p! \mu_y)$-Lipschitz continuous $p$th-order derivatives in $\vy$.    
    \item $h_{\epsilon}(\vx,\vy;\bar \vx, \bar \vy)$ is $(p+1)$th-order $((\gamma +\mu_x)/2^{p-1} )$-uniformly-convex-$((\gamma+\mu_y)/2^{p-1})$-uniformly-concave; has $(L_1 + p \max\{(\gamma+\mu_x) D_\gX^{p-1}, (\gamma+\mu_y) D_\gY^{p-1}\})$-Lipschitz continuous gradients and $(L_p + p! (\gamma+ \max\{\mu_x,\mu_y\}))$-Lipschitz continuous $p$th-order derivatives jointly in $(\vx,\vy)$.
    \item $\Phi(\vx)$ is $(p+1)$th-order $(\mu_x/ 2^{p-1})$-uniformly convex.
    \item $\Psi(\vy; \bar \vx)$ is $(p+1)$th-order $(\mu_y/ 2^{p-1})$-uniformly concave.
\end{enumerate}
\end{lemma}
Our proposed algorithm first applies Algorithm \ref{alg:ANPE-restart} on $\Phi(\vx)$, and then applies Algorithm \ref{alg:ANPE-restart} on $\Psi(\vy; \bar \vx)$.  It requires the (inexact) zeroth-order, first-order and high-order proximal oracles of $\Phi(\vx)$ and $\Psi(\vy; \bar \vx)$. We have discussed how to obtain the proximal oracles in Section \ref{subsec:algo}. Below, we show how to obtain the $\delta$-zeroth-order and first-order oracles for both $\Phi(\vx)$ and $\Psi(\vy; \bar \vx)$. 

\begin{lemma} \label{lem:get-zo-fo}
Under Assumption \ref{asm:D}, \ref{asm:CC}, \ref{asm:grad-lip}, and \ref{asm:Hess-lip}, OptMS-restart (Theorem \ref{thm:ANPE})
\begin{enumerate}
    \item finds a point $\hat \vy \in \gY$ such that $f_\epsilon(\vx,\hat \vy) - \Phi(\vx) \le \delta$ with complexity  
    \begin{equation} \label{eq:comp-zo-Phi}
        \gO \left( \left( \frac{L_p+p! \mu_y}{\mu_y} \right)^{2/(3p+1)} \log \left( \frac{D_\gY^2(L_1 + p \mu_y D_\gY^{p-1})}{\delta} \right) \right).
    \end{equation}
    \item finds a point $\hat \vy \in \gY$ such that $\Vert \nabla f_\epsilon(\vx,\hat \vy) - \nabla \Phi(\vx) \Vert \le \delta$ with complexity 
    \begin{equation} \label{eq:comp-fo-Phi}
        \gO \left( \left( \frac{L_p + p! \mu_y}{\mu_y} \right)^{2/(3p+1)} \log \left( \frac{D_\gY^2 \left(L_1+ p \mu_y D_\gY^{p-1} \right)^{p+2} }{\mu_y \delta^{p+1}} \right) \right).
    \end{equation} 
    \item finds a point $\hat \vx \in \gX$ such that $\Psi(\vy ; \bar \vx)- g_\epsilon(\hat \vx,\vy; \bar \vx) \le \delta$ with complexity
    \begin{equation} \label{eq:comp-zo-Psi}
        \gO \left( \left( \frac{L_p + p! (\gamma + \mu_x)}{\gamma + \mu_x} \right)^{2/(3p+1)} \log \left( \frac{D_\gX^2 (L_1+ p (\gamma + \mu_x) D_\gX^{p-1})}{\delta} \right) \right).
    \end{equation}
    \item finds a point $\hat \vx \in \gX$ such that $\Vert \nabla_y g_\epsilon(\hat \vx,\vy; \bar \vx) - \nabla \Psi(\vy ; \bar \vx) \Vert \le \delta$ with complexity 
    \begin{equation} \label{eq:comp-fo-Psi}
        \gO \left( \left( \frac{L_p + p!(\gamma+ \mu_x)}{\gamma+ \mu_x} \right) 
        \log \left( \frac{D_\gX^2 \left(L_1+ p (\gamma+\mu_x) D_\gX^{p-1} \right)^{p+2} }{\mu_x \delta^{p+1}} \right)
        \right).
    \end{equation}
\end{enumerate}
\end{lemma}

\begin{proof}
The uniform convexity (concavity) and Lipschitz constants of the objectives $f_\epsilon(\vx,\,\cdot\,)$ and $g_\epsilon(\,\cdot\, , \vy ; \bar \vx)$
are given in Lemma \ref{lem:prop-funcs}. Given these quantities, we can apply Theorem~\ref{thm:ANPE} to 
prove all the items in a similar manner. In the following, we only provide the proof of item 1 and 2 since item 3 and 4 are simply the dual statements of them.
\begin{enumerate}
    \item 
According to Theorem \ref{thm:ANPE}, the complexity of finding a point $\hat \vy \in \gY$ such that $f_\epsilon(\vx,\hat \vy) - \Phi(\vx) \le \delta$ using OptMS-restart is upper bounded by 
\begin{equation*}
\gO \left( \left( \frac{L_p+p! \mu_y}{ \mu_y} \right)^{2/(3p+1)} \log \left( \frac{f_\epsilon(\vx,\vy_0) - \Phi(\vx)}{\delta} \right) \right),   
\end{equation*}
where $\vy_0$ is the initialization point for the algorithm. We then obtain Eq. (\ref{eq:comp-zo-Phi}) by using 
\begin{equation*}
  f_\epsilon(\vx,\vy_0) - \Phi(\vx) \le (L_1 + p \mu_y D_\gY^{p-1}) \Vert \vy_0  - \vy^*(\vx) \Vert^2.
\end{equation*}
\item By Danskin's theorem, we have that $\nabla \Phi(\vx) = \nabla_x f_\epsilon(\vx,\vy^*(\vx))$, where $\vy^*(\vx) = \arg \max_{\vy \in \gY} f_\epsilon(\vx,\vy)$. Therefore, to obtain a $\delta$-first-order oracle of $\Phi(\vx)$ under Assumption \ref{asm:grad-lip}, it suffices to find $\hat \vy \in \gY$ such that 
\begin{equation*}
  \Vert \hat \vy - \vy^*(\vx) \Vert \le \frac{\delta}{L_1 + p \mu_y D_\gY^{p-1}} .  
\end{equation*} 
Further,
by Lemma \ref{lem:UC-grad-dominant} (the second inequality) and Lemma \ref{lem:reg-f-eps-close} (item 1), it only needs to find a point $\hat \vy \in \gY$ such that
\begin{equation*}
    f_{\epsilon}(\vx,\hat \vy) - \Phi(\vx) \le \frac{\mu_y}{2^{p-1}(p+1)} \left(\frac{\delta}{L_1 + p \mu_y D_\gY^{p-1}}  \right)^{p+1}.
\end{equation*}
Then we can apply item 1 to show item 2 for any constant $p \in \sN_+$.
\end{enumerate}
 
\end{proof}

Below, we will show that the above complexity of obtaining zeroth- and first-order oracles for both primal and dual functions in Lemma \ref{lem:get-zo-fo} is negligible compared to the complexity of obtaining their proximal oracles and then complete the proof of Theorem \ref{thm:Minimax-AIPE-CC}.

\begin{proof}[Proof of Theorem \ref{thm:Minimax-AIPE-CC}]
We analyze the complexity of each loop individually.
\noindent \textbf{Outer-loop complexity.} Let $\zeta_1 > 0 $ be the precision that will be chosen to be sufficiently small.
We first assume that, in each iteration, the sub-routine Algorithm~\ref{alg:Minimax-AIPE-mid} can return a $(\delta_1,\gamma)$ proximal oracle of $\Phi(\vx)$ such that $\delta_1 \lesssim \zeta_1^{2(p+2) / (p+1)} /\mu_x^{1/(p+1)}$. Let 
\begin{equation*}
    (\vx^*,\vy^*) = \arg \min_{\vx \in \gX} \max_{\vy \in \gY} f_\epsilon(\vx,\vy).
\end{equation*}
Then, by Theorem \ref{thm:AIPE} and Lemma \ref{lem:UC-grad-dominant} on $\Phi(\vx)$, we know that our algorithm finds 
$\Vert \hat \vx - \vx^* \Vert \le \zeta_1$ in 
\begin{equation} \label{eq:compl-outer}
   \# \text{OuterLoop} =  \gO\left( \left( \frac{\gamma}{\mu_x}\right)^{2/(3p+1)} \log \left( \frac{D_\gX^2 (L_1 + p \mu_x D_\gX^{p-1})}{\mu_x \zeta_1^{p+1}} \right) \right)
\end{equation}
calls of the sub-routine Algorithm \ref{alg:Minimax-AIPE-mid}. And the complexity of obtaining $\delta_1$-zeroth and first-order oracle of $\Phi(\vx)$ is given by Eq. (\ref{eq:comp-zo-Phi}) and Eq. (\ref{eq:comp-fo-Phi}) with $\delta = \delta_1$, respectively. Moreover, we know from Theorem \ref{thm:ANPE} that OptMS-restart can find $\hat \vy \in \gY $ such that $\Vert \hat \vy - \vy^*(\hat \vx) \Vert \le \zeta_1$, where $\vy^*(\vx) = \arg \max_{\vy \in \gY} f_\epsilon(\vx,\vy)$, in $p$th-order oracle complexity upper-bounded by
\begin{equation*}
    \gO \left( \left( \frac{L_p + p! \mu_y}{\mu_y} \right)^{2/(3p+1)} \log \left( \frac{D_\gY^2 (L_1 + p \mu_y D_\gY^{p-1}) }{\mu_y \zeta_1^{p+1}} \right) \right).
\end{equation*}
 Now, using $\vy^*(\vx^*) = \vy^*$, Lemma \ref{lem:reg-f-eps-close} (item 1), and Lemma \ref{lem:benigh-reg-f} (item 2) on $f_\epsilon(\vx,\vy)$, we have
\begin{align*}
\Vert \hat \vy - \vy^* \Vert \le& \Vert \hat \vy - \vy^*(\hat \vx) \Vert + \Vert \vy^*(\hat \vx)  - \vy^*(\vx^*) \Vert \\
\le&\Vert \hat \vy - \vy^*(\hat \vx) \Vert + \left(\frac{(p+1) 2^{p-2} L_1 \Vert \hat \vx - \vx^* \Vert}{ \mu_y} \right)^{1/p}. 
\end{align*}
This means that we have
\begin{equation} \label{eq:final-dist}
    \Vert \hat \vz - \vz^* \Vert \le \Vert \hat \vx - \vx^* \Vert+ \Vert \hat \vy - \vy^* \Vert\le 2\zeta_1 +  \left(\frac{ (p+1) 2^{p-2}  L_1 \zeta_1}{\mu_y} \right)^{1/p}.
\end{equation}
Let $\tilde L_1 = (L_1 + p \max\{ \mu_x D_\gX^{p-1}, \mu_y D_{\gY}^{p-1} \})$ be the gradient Lipschitz constant of $f_{\epsilon}(\vx,\vy)$ given by Lemma \ref{lem:prop-funcs} (item 1). Then Lemma \ref{lem:gd-make-g-small} shows that the algorithm outputs $\tilde \vz \in \gZ$   such that 
\begin{equation} \label{eq:final-gnorm}
    \left\Vert 
         \mF(\tilde \vz) + \tilde \vc
        \right \Vert \le  
        6 \tilde L_1 \Vert \hat \vz - \vz^* \Vert,
\end{equation}
where $\tilde \vc = \tilde L_1 (\hat \vz - \tilde \vz)- \mF(\hat \vz) \in \gN_{\gZ}(\tilde \vz)$.
Combining Eq. (\ref{eq:final-dist}) and (\ref{eq:final-gnorm}), we know there exists $\zeta_1 = 1 / {\rm poly}(\tilde L_1, \nicefrac{1}{\mu_y}, \nicefrac{1}{\epsilon})$ such that  $\Vert 
         \mF(\tilde \vz) + \tilde \vc
        \Vert \le \epsilon/2$. Finally, by Lemma \ref{lem:reg-f-eps-close} (item 2), we know the output is an $\epsilon$-solution to the original function $f(\vx,\vy)$.


\noindent \textbf{Middle-loop complexity}. 
Let $\zeta_2 > 0 $ be the precision that will be chosen to be sufficiently small. 
We first assume that, in each iteration, the sub-routine Algorithm \ref{alg:Minimax-AIPE-inner} can return a $(\delta_2,\gamma)$ proximal oracle of $\Psi(\vy;\bar \vx)$ such that $\delta_2 \lesssim \zeta_2^{2(p+2) / (p+1) } / \mu_y^{1/(p+1)}$.
Our goal is to show that the middle loop can return a $(\delta_1,\gamma)$ proximal oracle of $\Phi(\vx)$ with the precision $\delta_1$ required by the outer loop. Let
\begin{equation*}
   (\vx^*(\bar \vx), \vy^*(\bar \vx)) = \arg \min_{\vx \in \gX} \max_{\vy \in \gY} \left\{g_\epsilon(\vx,\vy; \bar \vx):= f_\epsilon(\vx,\vy) + \frac{\gamma}{p+1} \Vert \vx - \bar \vx \Vert^{p+1} \right\}.
\end{equation*} 
Then, by Theorem \ref{thm:AIPE} and Lemma \ref{lem:UC-grad-dominant} on $\Psi(\vy; \bar \vx)$, our algorithm finds 
$ \Vert \hat \vy - \vy^*(\bar \vx) \Vert \le \zeta_2$  in 
\begin{equation} \label{eq:compl-mid}
   \# \text{MiddleLoop} =  \gO\left( \left( \frac{\gamma}{\mu_y}\right)^{2/(3p+1)} \log \left( \frac{D_\gY^2 (L_1 + p \mu_y D_\gY^{p-1})}{\mu_y \zeta_2^{p+1}} \right) \right)
\end{equation}
calls of the sub-routine Algorithm \ref{alg:Minimax-AIPE-inner}. The complexity of obtaining $\delta_2$-zeroth and first-order oracle of $\Psi(\vy;\bar \vx)$ is given by Eq. (\ref{eq:comp-zo-Psi}) and Eq. (\ref{eq:comp-fo-Psi}) with $\delta = \delta_2$, respectively. Also,
according to Lemma~\ref{lem:UC-grad-dominant}, and Lemma \ref{lem:reg-f-eps-close} (item 1), and Lemma \ref{lem:benigh-reg-f} (item 1), we can apply
Lemma \ref{lem:get-zo-fo} (item~3) with $\delta = \mu_x \zeta_2^{p+1} / ((p+1)2^{p-1} )$ to conclude that OptMS-restart finds $\hat \vx \in \gX$ such that $\Vert \hat \vx - \vx^*(\hat \vy; \bar \vx) \Vert \le \zeta_2$, where $\vx^*(\vy; \bar \vx) = \arg \min_{\vx \in \gX} g_\epsilon(\vx,\vy; \bar \vx)$, in $p$th-order oracle complexity upper-bounded by 
\begin{equation*}
    \gO \left( \left( \frac{L_p + p! (\gamma +\mu_x)}{\gamma + \mu_x} \right)^{2/(3p+1)} \log \left( \frac{D_\gX^2 (L_1 + p (\gamma+\mu_x) D_\gX^{p-1})}{
    \mu_x \zeta_2^{p+1}} \right) \right).
\end{equation*}
Let $\tilde L_1^x = L_1 + p (\gamma+ \mu_x) D_\gX^{p-1}$ be the gradient Lipschitz constant of $g_{\epsilon}(\vx,\vy; \bar \vx)$ given by Lemma \ref{lem:prop-funcs} (item~2). Then by Lemma \ref{lem:gd-make-g-small}, our algorithm outputs $\tilde \vx \in \gX$ and $\tilde \vu \in \gN_{\gX}(\tilde \vx)$ such that 
\begin{equation} \label{eq:grad-small-g}
    \Vert \nabla_x g_\epsilon(\tilde \vx, \hat \vy; \bar \vx) + \tilde \vu \Vert  \le 6 \tilde L_1^x \Vert \hat \vx - \vx^*(\hat \vy; \bar \vx) \Vert \le 6 \tilde L_1^x \zeta_2.
\end{equation}
Our next goal is to show that Eq. (\ref{eq:grad-small-g}) implies an approximate proximal oracle of $\Phi(\vx)$ at the query point $\bar \vx \in \gX$ if $\zeta_2$ is sufficiently small.
Applying Lemma~\ref{lem:benigh-reg-f} (item 2) on $g_\epsilon(\vx,\vy;\bar \vx)$, we can obtain
\begin{align} \label{eq:cont-mid}
\begin{split}
    \Vert \vy^*(\vx_1; \bar \vx) - \vy^*(\vx_2; \bar \vx) \Vert^p &\le \frac{(p+1)2^{p-2} L_1}{ \mu_y} \Vert \vx_1 - \vx_2 \Vert, \quad \forall \vx_1,\vx_2 \in \gX; \\
    \Vert \vx^*(\vy_1; \bar \vx) - \vx^*(\vy_2; \bar \vx) \Vert^p &\le \frac{(p+1)2^{p-2} L_1}{(\gamma + \mu_x)} \Vert \vy_1 - \vy_2 \Vert, \quad \forall \vy_1,\vy_2 \in \gY,
\end{split}
\end{align}
where  $\vy^*(\vx; \bar \vx) =\arg \max_{\vy \in \gY} g_\epsilon(\vx,\vy; \bar\vx) $ and $\vx^*(\vy; \bar \vx) = \arg \min_{\vx \in \gX} g_\epsilon(\vx,\vy; \bar\vx)$. Let 
$\lambda = \gamma \Vert \tilde \vx - \bar \vx \Vert$, then we have
 \begin{align}  \label{eq:Phi-gnorm}
 \begin{split}
      &\Vert \nabla \Phi(\tilde \vx) + \lambda (\tilde \vx - \bar \vx) + \tilde \vu \Vert =
    \Vert \nabla_x g_\epsilon(\tilde \vx, \vy^*(\tilde \vx; \bar \vx) ;\bar \vx) + \tilde \vu \Vert \\
    \le& \Vert \nabla_x g_\epsilon(\tilde \vx, \hat \vy; \bar \vx) + \tilde \vu \Vert + \tilde L_1^x \left(\Vert \hat \vy - \vy^*(\bar \vx) \Vert + \Vert \vy^*(\bar \vx) - \vy^*(\tilde \vx; \bar \vx) \Vert \right) \\ 
    \le& \Vert \nabla_x g_\epsilon(\tilde \vx, \hat \vy; \bar \vx) + \tilde \vu \Vert + \tilde L_1^x \left(\Vert \hat \vy - \vy^*(\bar \vx) \Vert + \left(\frac{(p+1) 2^{p-2}  L_1 \Vert \tilde \vx - \vx^*(\bar \vx) \Vert }{\mu_y} \right)^{1/p} \right),
 \end{split}
\end{align}
where we use Eq. (\ref{eq:cont-mid}) and $\vy^*(\vx^*(\bar \vx)) = \vy^*(\bar \vx)$ in the last step. We want to let the right-hand side of Eq. (\ref{eq:Phi-gnorm}) be smaller than $\delta_1$, such that the sub-routine Algorithm \ref{alg:Minimax-AIPE-mid} returns a $(\delta_1,\gamma)$ proximal oracle of $\Phi(\vx)$ by Definition \ref{dfn:inexact-MS-oracle}. Since we have already shown that the first two terms can be arbitrarily close to zero when $\zeta_2 \rightarrow 0$, we only need to handle the third term involving $\Vert \tilde \vx - \vx^*(\bar \vx) \Vert$, which by Lemma \ref{lem:UC-grad-dominant}, Eq. (\ref{eq:cont-mid}), and the fact $\vx^*(\vy^*(\bar \vx)) = \vx^*(\bar \vx)$ can be upper-bounded by
\begin{align} \label{eq:Phi-xdist}
\begin{split}
     &\quad \Vert \tilde \vx - \vx^*(\bar \vx) \Vert \le \Vert \tilde \vx - \vx^*(\hat \vy; \bar \vx) \Vert +  \Vert \vx^*(\hat \vy; \bar \vx) - \vx^*(\bar \vx) \Vert   \\
    &\le p^{\frac{1}{p+1}} \left( \frac{(p+1)2^{p-1} \Vert \nabla_x g_\epsilon(\tilde \vx, \hat \vy; \bar \vx) + \tilde \vu \Vert} {\gamma + \mu_x}  \right)^{1/p} +
    \left( \frac{(p+1) 2^{p-2}  L_1 \Vert \hat \vy - \vy^*(\bar \vx) \Vert}{(\gamma + \mu_x)}  \right)^{1/p}.
\end{split}
\end{align}
Combining $\Vert \hat \vy - \vy^*(\bar \vx) \Vert \le \zeta_2 $,  Eq. (\ref{eq:grad-small-g}), and Eq. (\ref{eq:Phi-xdist}), we know that there exists some $\zeta_2 = 1/{\rm poly}( \tilde L_1^x, \nicefrac{1}{\mu_y}, \nicefrac{1}{(\gamma + \mu_x)}, \nicefrac{1}{\delta_1} )$  such that the right-hand side in Eq. (\ref{eq:Phi-gnorm})
is smaller than $\delta_1$, which 
fulfills the goal of the middle loop.
\noindent \textbf{Inner-loop complexity.} 
Let $\zeta_3 > 0 $ be the precision that will be chosen to be sufficiently small. 
Our goal is to show that the inner loop can return a $(\delta_2,\gamma)$-proximal of $\Psi(\vy;\bar \vx)$ with the precision $\delta_2$ required by the middle loop. Recall 
\begin{equation*}
     (\vx^*(\bar \vx, \bar \vy), \vy^*(\bar \vx, \bar \vy)) = \arg \min_{\vx \in \gX} \max_{\vy \in \gY} h_{\epsilon}(\vx,\vy; \bar \vx, \bar \vy),
\end{equation*}
where $h_\epsilon(\vx,\vy; \bar \vx, \bar \vy)$ is the two-sided regularized minimax function defined in Eq. (\ref{prob:minimax-h}).
Since each epoch of Algorithm \ref{alg:Minimax-AIPE-inner} is exactly the $p$th-order EG method \citep{bullins2022higher} on $h_{\epsilon}(\vx,\vy;\bar \vx,\bar \vy)$.
Let $\tilde L_p^\gamma = L_p + p! ( \gamma + \max\{\mu_x, \mu_y\})$ be
the $p$th-order Lipschitz constant of $h_\epsilon(\vx,\vy; \bar \vx, \bar \vy)$ 
given in Lemma \ref{lem:prop-funcs} (item 3).
Then, by \citep[Lemma 4.1]{bullins2022higher}, setting $M = 32 \tilde L_p^{\gamma}$ ensures that each epoch of the algorithm satisfies
\begin{equation*}
    \frac{1}{\sum_{t=0}^{T-1} \eta_t^{(s)}} \sum_{t=0}^{T-1} \left \langle
    \eta_t^{(s)} \mF^{\gamma}(\vz_t^{(s)}),  \vz_t^{(s)} -  \vz^*(\bar \vz)
    \right \rangle \le \frac{16 \tilde L_p^\gamma \Vert \vz^{(s)} - \vz^*(\bar \vz) \Vert^{p+1}}{p!~T^{(p+1)/2}}.
\end{equation*} 
Since $h_\epsilon(\vx,\vy;\bar \vx, \bar \vy)$ is $(p+1)$th-order $(\gamma / 2^{p-1})$-uniformly-convex-$ (\gamma / 2^{p-1})$-uniformly-concave by Lemma \ref{lem:prop-funcs} (item 3), 
we can apply 
Lemma~\ref{lem:U-monotone} on the left-hand side of the above inequality to obtain that
\begin{equation*}
    \Vert \vz^{(s+1)} - \vz^*(\bar \vz) \Vert^{p+1} \le \frac{2^{p+2} \tilde L_p \Vert \vz^{(s)} - \vz^*(\bar \vz) \Vert^{p+1}}{\gamma p!~T_3^{(p+1)/2}}.
\end{equation*}
For any constant $p \in \sN_+$, the above inequality means we can find $\Vert \vz^{(S_3)} - \vz^*(\bar \vz) \Vert \le \zeta_3$ by setting the epoch length $T_3 = \gO( (\nicefrac{\tilde L_p^\gamma}{\gamma})^{(p+1)/2})$ and number of epochs $S_3 = \gO( \log (\nicefrac{D_\gZ}{\zeta_3}))$. Hence, the $p$th-order oracle complexity of the inner loop is upper bounded by
\begin{equation} \label{eq:compl-inner}
    \# \text{InnerLoop} = \gO \left( \left( \frac{\tilde L_p^\gamma}{\gamma } \right)^{(p+1)/2} \log \left( \frac{D_\gZ}{\zeta_3} \right) \right).
\end{equation}
Let $\tilde L_1^\gamma = (L_1 + p \max\{(\gamma+\mu_x) D_\gX^{p-1}, (\gamma+\mu_y) D_\gY^{p-1}\})$ be 
the gradient Lipschitz constant of $h_\epsilon(\vx,\vy; \bar \vx, \bar \vy)$ 
given in Lemma \ref{lem:prop-funcs} (item 3).
Then Lemma \ref{lem:gd-make-g-small} shows that the algorithm outputs $\hat \vz = (\hat \vx, \hat \vy) \in \gZ$ and $\hat \vc =(\hat \vu, \hat \vv) \in \gN_\gZ(\hat \vz)$ such that
\begin{equation} \label{eq:grad-small-F-gamma}
   \left \Vert 
   \begin{matrix}
       \nabla_x h_\epsilon(\hat \vx,\hat \vy;\bar \vx, \bar \vy) + \hat \vu \\
       -\nabla_y h_\epsilon(\hat \vx,\hat \vy;\bar \vx, \bar \vy) + \hat \vv
   \end{matrix}
   \right \Vert =   
    \Vert \mF^{\gamma}(\hat \vz) + \hat \vc \Vert \le 6 \tilde L_1^\gamma \Vert \vz^{(S_3)} - \vz^*(\bar \vz) \Vert \le 6 \tilde L_1^\gamma \zeta_3.
\end{equation}
Our next goal is to show that Eq. (\ref{eq:grad-small-F-gamma}) implies an approximate proximal oracle of $\Psi(\vy ;\bar \vx)$ at the query point $\bar \vy \in \gY$ if $\zeta_3$ is
sufficiently small. By the form of $g_\epsilon(\vx,\vy; \bar \vx)$ and $h_\epsilon(\vx,\vy; \bar \vx, \bar \vy)$, we know that $\vx^*(\vy; \bar\vx) = \vx^*(\vy; \bar\vx , \bar \vy)$,
where   $ \vx^*(\vy; \bar\vx) = \arg \min_{\vx \in \gX} g_\epsilon(\vx,\vy; \bar \vx)$ and $ \vx^*(\vy; \bar\vx , \bar \vy) = \arg \min_{\vx \in \gX} h_\epsilon(\vx,\vy; \bar \vx, \bar \vy)$. 
Hence, we can use the second inequality of Eq. (\ref{eq:cont-mid}) to obtain
\begin{equation} \label{eq:cont-inner}
    \Vert  \vx^*(\vy_1; \bar\vx , \bar \vy) - \vx^*(\vy_2; \bar\vx , \bar \vy) \Vert^p \le \frac{(p+1) 2^{p-2} L_1}{\gamma + \mu_x} 
    \Vert \vy_1 - \vy_2 \Vert, \quad \forall \vy_1, \vy_2 \in \gY,
\end{equation}
Let $\lambda = \gamma \Vert \hat \vy - \bar \vy \Vert$. We then have
\begin{align} \label{eq:Psi-gnorm}
\begin{split}
      &\Vert \nabla \Psi(\hat \vy; \bar \vx) + \lambda (\tilde \vx - \bar \vx) + \hat \vv \Vert =
    \Vert \nabla_y h_\epsilon(\vx^*(\hat \vy;\bar \vx, \bar \vy), \hat \vy ;\bar \vx, \bar \vy) + \hat \vv \Vert \\
    \le& \Vert \nabla_y h_\epsilon(\hat \vx, \hat \vy; \bar \vx, \bar \vy) +  \hat \vv \Vert + \tilde L_1^\gamma \left(\Vert \hat \vx - \vx^*(\bar \vx, \bar \vy) \Vert  + \Vert \vx^*(\bar \vx, \bar \vy) - \vx^*(\hat \vy;\bar \vx, \bar \vy) \Vert \right) \\ 
    \le& \Vert \nabla_y h_\epsilon(\hat \vx, \hat \vy; \bar \vx, \bar \vy) +  \hat \vv \Vert + \tilde L_1^\gamma \left( \Vert \hat \vx - \vx^*(\bar \vx, \bar \vy) \Vert
    + \left( \frac{(p+1) 2^{p-2} L_1}{\gamma + \mu_x} \Vert \hat \vy - \vy^*(\bar \vx, \bar \vy) \Vert \right)^{1/p} \right),
\end{split}
\end{align}
where we use Eq. (\ref{eq:cont-inner}) and $ \vx^*( \vy^*(\bar \vx, \bar \vy); \bar \vx, \bar \vy) = \vy^*(\bar \vx, \bar \vy) ) $ in the last step. We want the right-hand side of the above inequality to be smaller than $\delta_2$, such that the sub-routine returns a $(\delta_2,\gamma)$ proximal oracle of $\Psi(\vy;\bar \vx)$ by Definition \ref{dfn:inexact-MS-oracle}. Since we have
already shown that the first term can be arbitrarily close to zero if $\zeta_3 \rightarrow 0$, we only need to handle the last two terms involving $\Vert \hat \vx - \vx^*(\bar \vx, \bar \vy) \Vert$ and $\Vert \hat \vy - \vy^*(\bar \vx, \bar \vy) \Vert$, which by Corollary~\ref{cor:U-monotone} can both be upper-bounded by 
\begin{equation}  \label{eq:Psi-zdist}
 \max \{\Vert \hat \vx - \vx^*(\bar \vx, \bar \vy) \Vert, \Vert \hat \vy - \vy^*(\bar \vx, \bar \vy) \Vert  \}  \le \Vert \hat \vz - \vz^*(\bar \vz) \Vert \le \left( \frac{(p+1) 2^{p-2}}{\gamma} \Vert \mF^\gamma(\hat \vz; \bar \vz) + \hat \vc \Vert \right)^{1/p}.
\end{equation}
Now, combining Eq. (\ref{eq:grad-small-F-gamma}) and (\ref{eq:Psi-zdist}),
we know that there exists some $\zeta_3 = 1/ {\rm poly}(\tilde L_1^\gamma, \nicefrac{1}{\gamma}, \nicefrac{1}{\delta_2})$ such that the right-hand side in Eq. (\ref{eq:Psi-gnorm}) is smaller than $\delta_2$  and thus
fulfills the goal of the inner loop.

\noindent \textbf{Total complexity.} 
Finally, the total complexity claimed in Eq. (\ref{eq:main-complexity}) can be easily obtained by multiplying Eq. (\ref{eq:compl-outer}), Eq. (\ref{eq:compl-mid}), and Eq. (\ref{eq:compl-inner}) under the setting $\gamma=L_p$.  
\end{proof}

\section{Proof of Lemma \ref{lem:pro-hard-ins}} \label{apx:lb-proof}

\begin{proof}
We prove the lemma by verifying each property one by one.

\begin{enumerate}
    \item  It is apparent that both the sets $\bar \gX$ and $\bar \gY$ are convex. We only need to show that $\bar f(\vx,\vy)$ is convex-concave in $\bar \gX \times \bar \gY$. Let 
    \begin{equation} \label{eq:B}
        \mB = 
        \begin{bmatrix}
            -1 & \\
            1 & -1 &   & \\
             & \ddots & \ddots & \\
             & & 1 & -1
        \end{bmatrix}
    \end{equation}
    and $\ve_1 = (1,0,\cdots,0)^\top$. Let $\bar g(\vx,\vy) = \sum_{i=1}^{T+1} \vy_{[i]} \vx_{[i]}^p$. This function is convex-concave on the domain $\sR^{T+1}_+ \times \sR^{T+1}_+$. Since $\bar f(\vx,\vy) = (L_p / (2^{p+1} p!)) \bar g(\mB \vx + \ve_1, \vy)$ and the affine mapping is a convexity–preserving operation, we know that $\bar f(\vx,\vy)$ is also convex-concave on the domain $ \{\vx: \mB \vx + \ve_1 \in \sR^{T+1}_+ \} \times \sR^{T+1}_+$. Next, we show that $ \bar f(\vx,\vy) $ has $L_p$-Lipschitz continuous $p$th-order derivatives.
    It is easy to show that
    \begin{equation*}
        \Vert \nabla^p \bar g(\vx,\vy) - \nabla^p \bar g(\vx',\vy') \Vert \le 2p! (\Vert \vx - \vx' \Vert + \Vert \vy - \vy' \Vert), \quad \forall \vx,\vx,\vy,\vy' \in \sR^{T+1}.
    \end{equation*}
    Since the matrix $\mB$ in Eq. (\ref{eq:B}) satisfies $\Vert \mB \Vert \le 2$, we further have that
    \begin{equation*}
        \Vert \nabla^p \bar f(\vx,\vy) - \nabla^p \bar f(\vx',\vy') \Vert \le L_p \Vert \vx - \vx' \Vert + \Vert \vy - \vy' \Vert), 
        \quad \forall \vx,\vx,\vy,\vy' \in \sR^{T+1}.
    \end{equation*}
    
    \item 
    We prove the property in Eq. (\ref{eq:zero-respect}) by induction. We let
    \begin{equation*}
        \sR_t^{T+1} = \{ \vx \in \sR^{T+1}: \vx_{[i]} =0, ~~ i = t+1,\cdots,T+1 \}.
    \end{equation*}
    Suppose $( \vx_t,\vy_t) \in \sR^{T+1}_t \times \sR^{T+1}_t
    $, our goal is to show that $(\vx_{t+1},\vy_{t+1}) \in (\sR^{T+1}_{t+1} \cap \bar \gX) \times (\sR^{T+1}_{t+1} \cap \bar \gY)
    $ when the algorithm lies in the class of Definition \ref{dfn:ZR-alg}. Note that our hard instance in Eq. (\ref{eq:our-hard-ins}) is constructed such that
    for any $q \in \{1,\cdots,p \}$ we have 
    \begin{equation*}
        \bar \mF^{q-1} (\bar \vz_t)[\vh]^q \in \sR^{T+1}_{t+1} \times  \sR^{T+1}_{t+1}, \quad \forall \vh \in \sR^{T+1} \times \sR^{T+1},
    \end{equation*}
Next, we consider the different cases when the algorithm takes different options. 

\vspace{4mm}

\textbf{Option (A).} We have $\vy_{t+1} = \vy_t$ and $\vx_{t+1}$ satisfies that
 \begin{equation*}
    \left \langle \sum_{k=1}^q \frac{1}{k!} (\nabla_x )^k \bar f(\bar \vx_t, \bar \vy_t) + \frac{M_t}{q!} \Vert \vx_{t+1} - \bar \vx_t \Vert (\vx_{t+1} - \bar \vx_t), \vx - \vx_{t+1}     \right \rangle \ge 0, ~~ \forall \vx \in \bar \gX.
\end{equation*}
In this case, we can prove that $\vx_{t+1} \in \sR_{t}^{T+1}$. Otherwise, $\vx_{t+1} \notin \sR_{t}^{T+1}$, and there must exist an index $j \in \{t+1,\cdots,T+1\}$ such that $(\vx_{t+1})_{[j]} > 0$. Then, letting $\vx = \vx_{t+1} - (\vx_{t+1})_{[j]} \ve_j$ in the above inequality leads to a contradiction.

\vspace{4mm}

\textbf{Option (B).} We have $\vx_{t+1} = \vx_t$ and $\vy_{t+1}$ satisfies that
\begin{equation*} 
    \left \langle \sum_{k=1}^q \frac{1}{k!} (-1)^k (\nabla_y )^k \bar f(\bar \vx_t, \bar \vy_t) + \frac{M_t}{q!} \Vert \vy_{t+1} - \bar \vy_t \Vert (\vy_{t+1} - \bar \vy_t), \vy - \vy_{t+1}     \right \rangle \ge 0, ~~ \forall \vy \in \bar \gY.
\end{equation*}
In this case, we can prove that $\vy_{t+1} \in \sR^{T+1}_{t+1}$. Otherwise $\vy_{t+1} \notin \sR^{T+1}_{t+1}$, there must exist an index $j \in \{ t+2,\cdots,T+1 \}$ such that $(\vy_{t+1})_{[j]} >0$. Define $j_m = \{ \max j: (\vy_{t+1})_{[j]} >0, ~j = t+2,\cdots T+1\}$. Then letting $\vy = \vy_{t+1} -(\vy_{t+1})_{[j_m]} \ve_{j_m} $ in the above inequality also leads to a contradiction.

\vspace{4mm}

\textbf{Option (C).} We have $\vz_{t+1} = (\vx_{t+1}, \vy_{t+1})$ satisfies that
\begin{equation*}
    \left \langle \sum_{k=0}^{q-1} \frac{1}{k!} \nabla^k \bar \mF(\bar \vz_t) + \frac{M_t}{q!} \Vert \vz_{t+1} - \bar \vz_t \Vert (\vz_{t+1} - \bar \vz_t), \vz - \vz_{t+1}     \right \rangle \ge 0, ~~ \forall \vz \in \bar \gZ.
\end{equation*}
We can mimic the analysis in option (A) for variable $\vx$ and the analysis in option ((B) for variable $\vy$ separately to show that $(\vx_{t+1},\vy_{t+1}) \in \sR^{T+1}_{t+1} \times \sR^{T+1}_{t+1} $.

Combining all the above three cases proves the induction that $(\vx_{t+1},\vy_{t+1}) \in (\sR^{T+1}_{t+1} \cap \gX) \times (\sR^{T+1}_{t+1} \cap \gY)$. And the induction base is trivially satisfied as $(\vx_0,\vy_0) = (\vzero,\vzero)$.
    
\item For any $\vz = (\vx,\vy) \in \bar \gX \times \bar \gY$  with $ \vx_{[T+1]} = \vy_{[T+1]} = 0 $, we can use Lemma \ref{lem:solu-concept} to give a lower bound of the tangent residual via duality gap as follows:
\begin{align*}
   D_\gZ r_{\bar \mF}^{\rm tan}(\vz) & \ge \frac{L_p}{2^{p+1} p!} \left(\max_{\vy' \in \gY} \bar f(\vx,\vy') - \min_{\vx' \in \gX} \bar f(\vx',\vy) \right) \\
   &\ge \frac{L_p}{2^{p+1} p!}  \left( \bar f(\vx, \vone) - \bar f(\vone, \vy) \right) = \frac{L_p}{2^{p+1} p!} \bar f(\vx_T, \vone) \\
   &= \frac{L_p}{2^{p+1} p!} \left( \left(1- \vx_{[1]} \right)^p + \sum_{i=1}^{T-1} \left( \vx_{[i]} -\vx_{[i+1]}  \right)^p + \left(\vx_{[T]} \right)^p \right) \\
   &\ge \frac{L_p (T+1)}{2^{p+1} p!}  \left( \frac{1}{T+1} \left( 1- \vx_{[1]} +  \sum_{i=1}^{T-1}\left( \vx_{[i]} -\vx_{[i+1]}  \right) + \vx_{[T]}  \right) \right)^p \\
   &= \frac{L_p}{2^{p+1} p! (T+1)^{p-1}} = \frac{L_p \left(D_{\bar \gZ} / \sqrt{2} \right)^{p+1}}{2^{p+1} p! (T+1)^{(3p-1)/2}},
\end{align*}
where the last inequality uses the convexity of $x^p$ on $\sR_+$ and the last equality uses the fact that $D_{\bar \gX} = D_{\bar \gY} = \sqrt{T+1}$ and $D_{\bar \gZ} = \sqrt{2(T+1)}$.
\end{enumerate}
 
\end{proof}

\end{document}

%% file: math.tex
\usepackage{amsmath,amsfonts,bm,amssymb}

\def\eqref#1{equation~\ref{#1}}

\def\1{\bm{1}}

\def\vzero{{\bm{0}}}
\def\vone{{\bm{1}}}

\def\vc{{\bm{c}}}

\def\ve{{\bm{e}}}

\def\vg{{\bm{g}}}
\def\vh{{\bm{h}}}

\def\vu{{\bm{u}}}
\def\vv{{\bm{v}}}

\def\vx{{\bm{x}}}
\def\vy{{\bm{y}}}
\def\vz{{\bm{z}}}

\def\mA{{\bm{A}}}
\def\mB{{\bm{B}}}

\def\mF{{\bm{F}}}

\def\mM{{\bm{M}}}

\def\mT{{\bm{T}}}

\DeclareMathAlphabet{\mathsfit}{\encodingdefault}{\sfdefault}{m}{sl}
\SetMathAlphabet{\mathsfit}{bold}{\encodingdefault}{\sfdefault}{bx}{n}

\def\gA{{\mathcal{A}}}

\def\gI{{\mathcal{I}}}

\def\gM{{\mathcal{M}}}
\def\gN{{\mathcal{N}}}
\def\gO{{\mathcal{O}}}

\def\gT{{\mathcal{T}}}

\def\gX{{\mathcal{X}}}
\def\gY{{\mathcal{Y}}}
\def\gZ{{\mathcal{Z}}}

\def\sN{{\mathbb{N}}}

\def\sR{{\mathbb{R}}}

\usepackage{xcolor}
\usepackage{nicefrac}
\usepackage{cancel}
\newtheorem{theorem}{Theorem}[section]
\newtheorem{dfn}{Definition}[section]

\newtheorem{lemma}{Lemma}[section]
\newtheorem{asm}{Assumption}[section]
\newtheorem{remark}{Remark}[section]
\newtheorem{cor}{Corollary}[section]